\documentclass{article}[12pt]
\usepackage{latexsym}
\usepackage{amsmath}
\usepackage{amsfonts}
\usepackage{amssymb}
\usepackage{graphicx}
\usepackage{amsthm}
\usepackage[top=2.5cm, bottom=2.5cm, left=2.5cm, right=2.5cm]{geometry}
\usepackage{hyperref}
\usepackage{color}

\newcommand{\beq}{\begin{equation}}
\newcommand{\eeq}{\end{equation}}



\theoremstyle{plain}
\newtheorem{theorem}{Theorem}[section]
\newtheorem{proposition}[theorem]{Proposition}
\newtheorem{lemma}[theorem]{Lemma}

\theoremstyle{definition}

\newtheorem{remark}[theorem]{Remark}

\numberwithin{equation}{section}

\allowdisplaybreaks[4] 

\swapnumbers
\begin{document}

\title{On the initial boundary value problem for the Einstein vacuum equations in the maximal gauge}

\author{Grigorios Fournodavlos,\footnote{Sorbonne Universit\'e, CNRS, Universit\'e  de Paris, Laboratoire Jacques-Louis Lions (LJLL), F-75005 Paris, France,
		grigorios.fournodavlos@sorbonne-universite.fr}\;\; Jacques Smulevici\footnote{Sorbonne Universit\'e, CNRS, Universit\'e  de Paris, Laboratoire Jacques-Louis Lions (LJLL), F-75005 Paris, France, jacques.smulevici@sorbonne-universite.fr}}


\date{}

\maketitle

\parskip = 0 pt


\abstract{We consider the initial boundary value problem for the Einstein vacuum equations in the maximal gauge, or more generally, in a gauge where the mean curvature of a timelike foliation is fixed  {near the boundary}. We prove the existence of solutions such that the normal to the boundary is tangent to the time slices, the lapse of the induced time coordinate on the boundary is fixed and the main geometric boundary conditions are given by the $1$-parameter family of Riemannian conformal metrics on each two-dimensional section. As in the local existence theory of Christodoulou-Klainerman for the Einstein vacuum equations in the maximal gauge, we use as a reduced system the wave equations satisfied by the components of the second fundamental form of the the time foliation. The main difficulty lies in completing the above set of boundary conditions such that the reduced system is well-posed, but still allows for the recovery of the Einstein equations. We solve this problem by imposing the momentum constraint equations on the boundary, suitably modified by quantities vanishing in the maximal gauge setting. To derive energy estimates for the reduced system at time $t$, we show that all the terms in the flux integrals on the boundary can be either directly controlled by the boundary conditions or they lead to an integral on the two-dimensional section at time $t$ of the boundary. Exploiting again the maximal gauge condition on the boundary, this contribution  {to} the flux integrals can then be absorbed by  {a careful} trace inequality in the interior energy. }

\section{Introduction}
One of the most fundamental properties of the Einstein equations, if not the most, is their hyperbolic nature. This naturally leads to the initial value problem in General Relativity, solved in the work of Choquet-Bruhat \cite{Choq}: 
In the vacuum case, given any Riemannian manifold $(\Sigma,  {h})$ and covariant symmetric $2$-tensor $k$, satisfying the constraint equations 
\begin{eqnarray}
	R - | k |^2 + (\mathrm{tr} k)^2&=&0 \label{eq:hcons} \\
	d\,\mathrm{tr}k - \mathrm{div} k &=& 0, \label{eq:mcons}
	\end{eqnarray}
where $R$ is the scalar curvature of $h$, $\mathrm{tr} k$ is the trace of $k$ with respect to $h$, $d\,\mathrm{tr}k$ its differential and $\mathrm{div} k$ its divergence, there exists a maximal\footnote{The construction of a maximal development is due to Choquet-Bruhat and Geroch \cite{ChoqGer}.} globally hyperbolic development $(\mathcal{M},\bold{g})$ of $(\Sigma,h,k)$, which is unique modulo diffeomorphisms. 
As  {a} development of the data, it is a solution to the vacuum Einstein equations  {(EVE):
\begin{align}\label{EVE}
\mathrm{Ric}(\bold{g})=0, 
\end{align}
}such that $(\Sigma, h, k)$ can be embedded in $\mathcal{M}$ with $(h,k)$ being the first and second fundamental forms of the embedding.  {Due to the geometric nature of the equations}, to prove local well-posedness, one typically makes various gauge choices to derive a \emph{reduced system} of partial differential equations. This reduced system needs to be both hyperbolic in some sense and allow for the recovery of the full Einstein equations. 

In this paper, we will consider the initial boundary value problem, that is to say we are interested in constructing a $3+1$ Lorentzian manifold $(\mathcal{M},\bold{g})$ such that $\partial \mathcal{M}=\Sigma \cup \mathcal{T}$, where 
$\Sigma$ is a spacelike hypersurface of $\mathcal{M}$ with boundary $S$, $\mathcal{T}$ is a timelike hypersurface of $\mathcal{M}$ with compact boundary $S$ and $\Sigma \cap \mathcal{T}= S$. $\Sigma$ can be thought of as the initial hypersurface and thus, we should consider the first and second fundamental forms of $\Sigma$ as given, while on $\mathcal{T}$, boundary data or boundary conditions will need to be imposed so as to make the problem well-posed. On $S$, one typically needs compatibility conditions between the initial data and the boundary data. 

On top of its intrinsic mathematical appeal, the initial boundary value problem is motivated by
\begin{itemize}
\item the study of asymptotitcally Anti-de-Sitter spacetimes, which naturally leads to an initial boundary value problem after conformal rescaling, 
\item numerical applications, where, for numerical purposes, one typically needs to solve the equations on a finite domain with boundary, 
\item possible coupling with massive matter  {of} compact support, as for instance in the study of the Einstein-Euler equations, where the exterior region will posess such a timelike boundary \cite{KindEhlers}. 
\end{itemize}
The initial boundary value problem in General Relativity was first solved in the work of Friedrich in the Anti-de-Sitter case \cite{Fried} and Friedrich and Nagy for the  Einstein vacuum equations with timelike boundary \cite{FriedNag}. For an extensive review, we refer to \cite{SarTig}. Apart from the Friedrich-Nagy approach, which is based on the Bianchi equations and the construction of a special frame adapted to the boundary, the other well-developed theory for the study of the initial boundary value problem is that of Kreiss-Reula-Sarback-Winicour \cite{KRSW} based on generalized harmonic coordinates. 

In this paper, we prove well-posedness of the initial boundary value problem for the Einstein vacuum equations formulated in the maximal gauge, or more generally, in any gauge where the mean curvature of a timelike foliation is fixed. More precisely, we prove the existence of solutions such that the time slices intersect the boundary orthogonally, the lapse of the induced time coordinate on the boundary is fixed and the main geometric boundary conditions are given by the $1$-parameter family of Riemannian conformal metrics on each two-dimensional section. The dynamical variables that we consider are the components of the second fundamental form of the time foliation which satisfy a system of wave equations, as originally identified by Choquet-Bruhat-Ruggeri \cite{CBR}. In the work of Christodoulou-Klainerman \cite{CK}, this wave formulation of the equations was exploited to prove local existence of solutions to the Einstein equations in the maximal gauge. In the presence of a timelike boundary, one now needs to provide boundary conditions for the components of the second fundamental form. 
As is standard for geometric hyperbolic partial differential equations with constraints, the boundary conditions have to be compatible with the constraints.

We identify that these can be chosen as follows: 
\begin{itemize}
\item With $t$ as the time function, whose level sets are the maximal slices $\Sigma_t$, and $S_t=\Sigma_t \cap \mathcal{T}$, $\mathcal{T}$ being the timelike boundary, prescribing the conformal class of the induced metrics on each $S_t$ implies Dirichlet type boundary conditions for the traceless part of the projection of $k$ on each $S_t$. This essentially encodes the standard degrees of freedom corresponding to gravitational radiation. 
\item These boundary conditions are first complemented by the requirement that the slices $\Sigma_t$ intersect $\mathcal{T}$ orthogonally, by fixing the lapse of the induced time coordinate on $\mathcal{T}$ and by imposing the maximal condition $\mathrm{tr}k=0$ on $\mathcal{T}$. 
\item If $A,B$ are indices that correspond to spatial directions tangent to $\mathcal{T}$ and orthogonal to $\partial_ t$, it remains to impose boundary conditions on the trace of $k_{AB}$ (or equivalently on the volume forms of the induced metrics on $S_t$ ), as well as on $k_{NA}$, where $N$ denotes the unit normal to the boundary. For this, we identify a system of boundary conditions which is essentially equivalent to the momentum constraint equations \eqref{eq:mcons} in the maximal gauge, see \eqref{kNNbdcond}-\eqref{kCCbdcond}. 
\end{itemize}
The fact that, with these boundary conditions, on one hand, one can close energy estimates and on the other hand, one can a posteriori recover the Einstein equations, is the main contribution of this paper. 

Since the maximal gauge was instrumental in several global in time results for the Einstein equations, such as the monumental work of Christodoulou-Klainerman \cite{CK} on the stability of Minkowski space, our results may find applications in the global analysis of solutions in the presence of a timelike boundary. Moreover, let us mention that the BSSN formulation  {\cite{BSSN}}, which is heavily used in numerical analysis, is based on a $3+1$ decomposition of the Lorentzian metric and thus its analysis is likely to be closely related to the one we pursue here. 


One of the outstanding issues remaining, concerning the initial boundary value problem, is the geometric uniqueness problem of Friedrich \cite{Fried}. Apart from the AdS case, all other results establishing well-posedness for some formulations of the initial boundary value problem impose certain gauge conditions on the boundary and the boundary data depends on these choices. In particular, given a solution $(\mathcal{M},\bold{g})$ to the Einstein equations with a timelike boundary, different gauge choices will lead to different boundary data, in each of the formulations for which well-posedness is known. On the other hand, if we had been given the different boundary data a priori, we would not know that these lead to the same solution. The situation is thus different from the usual initial value problem for which only isometric data leads to isometric solutions, which one then regards as the same solution. In the AdS case, this problem admits one solution: in \cite{Fried}, Friedrich proved that one can take the conformal metric of the boundary as boundary data, which is a geometric condition independent of any gauge choice. 

The work of this paper still requires certain gauge conditions to be fixed, however, our boundary conditions describe at least part of the geometry of the boundary (via the family of conformal metrics). 



To state more precisely our main result, let us consider a Lorentzian manifold $(\mathcal{M},\bold{g})$ with a  {time} function $t$, such that 
$$
\bold{g}= - \Phi^2 dt^2 + g_{ij} dx^i dx^j,
$$
 {where $x^1,x^2,x^3$ are $t$-transported coodinates, $g,k$ denote the first and second fundamental forms of the level sets $\Sigma_t$ of $t$,} satisfying  {moreover the maximal condition} $\mathrm{tr} k =0$. 
We assume that we are given initial data  {$(h,k)$} on $\Sigma_0$ and that $\partial \mathcal{M}=\Sigma_0 \cup \mathcal{T}$, where $\mathcal{T}$ is a timelike boundary, which we assume coincides with $\{ x^3=0 \}$. Here, $\mathcal{M}$ admits coordinates $(t,x^1, x^2, x^3)$, where $x^3$ is assumed to be a boundary definining function. 
The induced metric on the boundary is thus given by 
$$
H=-\Phi^2 dt^2 + g_{11}(dx^1)^2 + g_{22}(dx^2)^2+ 2 g_{12}dx^1 dx^2. 
$$
The intersection of $\Sigma_t$ and $\mathcal{T}$ is a spacelike $2$-surface, denoted by $S_t$, with metric 
$$
q_t= g_{11}(dx^1)^2 + g_{22}(dx^2)^2+ 2 g_{12}dx^1 dx^2.
$$
Our boundary conditions will be such that we fix $\Phi=1$, as well as the conformal classes $[q_t]$ of the $1$-parameter family of metrics $q_t$. Since  {the cross sections} $S_t$ are all diffeomorphic to each other, one can think of this part of the data as a 1-parameter family of conformal metrics on a fixed $2$-dimensional manifold $S$. Moreover, on $S_0$, compatibility conditions between $q_t$ and $h,k$ will be required. These are introduced below in Subsection \ref{subsec:bc}, see \eqref{compcond}.

With this notation, we prove the following theorem
\begin{theorem}\label{mainthm}
Let $(\Sigma, h, k)$ be initial data satisfying the constraint equations  {\eqref{eq:hcons}-\eqref{eq:mcons}}, with $\Sigma$ a $3$-dimensional manifold with compact boundary $S$ and $k$ being traceless $tr k =0$  {(near the boundary)}. Consider a 1-parameter family of conformal metrics $[q_t]_{t \in I}$ on $S$, verifiying the compabitility conditions discussed in Subsection \ref{subsec:bc}. Then, there exists a Lorentzian manifold $(\mathcal{M},\bold{g})$ with timelike boundary $\mathcal{T}$,  {satisfying \eqref{EVE},} such that $\mathcal{M}$ is foliated by Cauchy hypersurfaces $\Sigma_t$, $t \in I$, an embedding of $\Sigma$ onto $\Sigma_0$ such that $(h,k)$ coincides with the first and second fundamental forms of the embedding, and the boundary conditions are verified on $\mathcal{T}$, as introduced above. The time interval of existence, $I$, depends continuously on the initial and boundary data.
\end{theorem}
 {
\begin{remark}
The conditions $\text{tr}k=0$, $\Phi\big|_{\mathcal{T}}=1$, are not essential for our overall local existence argument. However, we do need to fix the foliation $\Sigma_t$ (near the boundary), such that $\text{tr}k,\Phi\big|_{\mathcal{T}}$ are prescribed, sufficiently regular functions.\footnote{ {Conditions for the existence of spacelike foliations with prescribed mean curvature, in the presence of boundaries, have been established by Bartnik \cite{Bar}.}}
\end{remark}
\begin{remark}
	If the data is asymptotically flat and $\text{tr}k=0$ initially, then one can obtain a solution which is globally foliated by maximal hypersurfaces. 
\end{remark}
\begin{remark}
The spacetime metric ${\bf g}$ is constructed by solving a set of reduced equations in the above gauge, verifying appropriate boundary conditions, see Section \ref{sec:enest}. Uniqueness for these equations also holds, however, it does not imply the desired geometric uniqueness we would like to have for the EVE. 
\end{remark}

The proof of Theorem \ref{mainthm} is based on deriving energy estimates, near the boundary, for a system of reduced equations, subject to certain boundary conditions, that we set up in Section \ref{sec:fram}. The energy argument is carried out in Section \ref{sec:enest}. 
We should note that even in the case where no timelike boundary is involved, the most naive scheme based on  {standard} energy estimates would fail to close due to loss of derivatives, see Remark \ref{rem:loss}. 

Interestingly, for the boundary value problem, even after exploiting all our boundary conditions, the total energy flux of $k$ at the boundary does not a priori have a sign. However, we demonstrate that after a careful use of trace inequalities, the terms which a priori could have the wrong sign can be absorbed in the interior energies. Emphasis is given to one particular boundary term, which is at the level of the main top order energies, and which requires a certain splitting in order to be absorbed in the left hand side of the estimates, see Remark \ref{rem:trace}. Such a manipulation is possible thanks to the maximal condition being valid on the boundary.\footnote{We adopt $\text{tr}k=0$ as one of our boundary conditions, see Subsection \ref{subsec:bc}.}

In Section \ref{sec:verEVE}, we confirm that the solution to the reduced system of equations is in fact a solution to \eqref{EVE}. First, we derive a system of propagation equations for the Einstein tensor, subject to boundary conditions induced from the ones of the reduced system, which are eligible to an energy estimate. Combining this fact with the vanishing of the Einstein tensor initially and the homogeneity of the induced boundary conditions for the final system of equations, we infer its vanishing everywhere.

For the energy estimates of \ref{sec:enest}, in order to preserve the choice of boundary conditions, we commute the equations only by tangential derivatives and recover the missing normal derivatives from the equations. A similar argument is used in the recovery of the Einstein equations when commuting the equation for $\text{tr} k$ in Section \ref{sec:verEVE}. 

\begin{quote}
\textbf{Acknowledgements.} {\small We would like to thank everyone at the Mittag-Leffler Institute for their hospitality and for generating a stimulating atmosphere, which was very beneficial for the completion of this work, during our visit in the Fall 2019.
Both authors are supported by the \texttt{ERC grant 714408 GEOWAKI}, under the European Union's Horizon 2020 research and innovation program.}
\end{quote}

}
\section{Framework}\label{sec:fram}

\indent

Our framework of choice is the one used for proving local existence for the EVE in the original Christodoulou-Klainerman stability of Minkowski proof \cite{CK}. We include a detailed outline of the whole procedure for the sake of completeness,  {cf.} \cite[\S10.2]{CK}. Moreover, given that our main point of interest is the initial boundary value problem (IBVP), we will focus mostly on controlling the boundary terms arising in the local existence argument, both in the estimates for the reduced equations (Section \ref{sec:enest}) and in the recovery of the Einstein vacuum equations  (Section \ref{sec:verEVE}). 

Let $(\mathcal{M}^{1+3},{\bf g})$ be a Lorentzian manifold with a timelike, $1+2$  {dimensional}, boundary $\partial \mathcal{M}$.
We consider a time function $t$ and the associated vector field $\partial_t$, which is parallel to the gradient of $t$ and satisfies $\partial_t(t)=1$. Also,
let $x^1,x^2,x^3$ denote Lie transported coordinates along the integral curves of $\partial_t$. In this case the spacetime metric takes the form
\begin{align}\label{metric}
{\bf g}=-\Phi^2dt^2+ {g}=-\Phi^2dt^2+ {g}_{ij}dx^idx^j,&&\Phi=(-g^{\alpha\beta}\partial_\alpha t\partial_\beta t)^{-\frac{1}{2}},
\end{align}
where $\Phi$ is the lapse of the foliation $\{t=$const.$\}=:\Sigma_t$.
In this framework, the first variation equations read
\begin{align}\label{1stvar}
\partial_t {g}_{ij}=-2\Phi k_{ij},&&k_{ij}:={\bf g}(D_{\partial_i}\partial_j,e_0)=k_{ji},\quad e_0=\Phi^{-1}\partial_t,
\end{align}
where $D$ is the covariant derivative intrinsic of ${\bf g}$. 
The 2-tensor $k_{ij}$ is the second fundamental form of $\Sigma_t$. We also have 
\begin{align}\label{1stvarinv}
\partial_t {g}^{ij}=2\Phi k^{ij}.
\end{align}
Moreover, the second variation equations read
\begin{align}\label{2ndvar}
\partial_tk_{ij}= -{\nabla}_i {\nabla}_j\Phi+\Phi( {R}_{ij}+k_{ij}\mathrm{tr}k-2{k_i}^lk_{jl})-\Phi {\bf R}_{ij},
\end{align}
where $ {\nabla}, {R}_{ij}$ are the covariant connection and Ricci tensor of $ {g}$, while ${\bf R}_{ij}$ is the Ricci tensor of ${\bf g}$. Imposing the EVE, the latter vanishes, whereas the former equals \cite[(3.4.5)]{Wald}:
\begin{align}\label{Ricci}
\text{R}_{ij}(g)=&\,\partial_a\Gamma^a_{ji}-\partial_j\Gamma^a_{ia}+\Gamma^a_{ab}\Gamma_{ji}^b-\Gamma^a_{jb}\Gamma^b_{ai}\\
\notag=&\,\nabla_a\Gamma_{ji}^a-\nabla_j\Gamma_{ia}^a-\Gamma^a_{ab}\Gamma_{ji}^b+\Gamma^a_{jb}\Gamma^b_{ai}
\end{align}
where $\nabla \Gamma$ is interpreted tensorially, e.g., $$\nabla_a\Gamma_{ji}^a:=\partial_a\Gamma_{ji}^a+\Gamma^a_{ab}\Gamma^b_{ij}-\Gamma_{aj}^b\Gamma_{bi}^a-\Gamma_{ai}^b\Gamma_{jb}^a.$$
In order to reveal the hyperbolic structure of \eqref{1stvar}-\eqref{2ndvar}, we need to differentiate \eqref{2ndvar} in $\partial_t$ and work with its second order analogue.
\begin{proposition}\label{prop:equiveq}
Let ${\bf g}$ be a Lorentzian metric expressed in the above framework. Then the propagation equation
\begin{align}\label{e0Rij4}
e_0{\bf R}_{ij}=\nabla_i\mathcal{G}_j+\nabla_j\mathcal{G}_i-\nabla_i\nabla_j\mathrm{tr}k,&&\mathcal{G}_i:={\bf R}_{0i},
\end{align}
is equivalent to the following wave equation for $k_{ij}$:
\begin{align}\label{boxk}
&\notag e_0^2k_{ij}-\Delta_g k_{ij}\\
=&\,\Phi^{-3}\partial_t\Phi\nabla_i\nabla_j\Phi-\Phi^{-2}\nabla_i\nabla_j\partial_t\Phi+\Phi^{-2}\partial_t\Gamma^l_{ij}\partial_l\Phi+e_0(k_{ij}\mathrm{tr}k-2{k_i}^lk_{jl})\\
\notag&+\Phi^{-1}k_{ij}\Delta_g\Phi
-\Phi^{-1}k_i{}^a\nabla_a\nabla_j\Phi-\Phi^{-1}k_j{}^a\nabla_a\nabla_i\Phi\\
\notag&-\Phi^{-1}\nabla^a\Phi(\nabla_jk_{ia}+\nabla_ik_{ja}-2\nabla_ak_{ij})
+\Phi^{-1}\mathrm{tr}k\nabla_j\nabla_i\Phi\\
\notag&+\Phi^{-1}\nabla_j\Phi(\nabla_i\mathrm{tr}k-\nabla_ak_i{}^a)
+\Phi^{-1}\nabla_i\Phi(\nabla_j\mathrm{tr}k-\nabla_ak_j{}^a)\\
\notag&-3(k_{ci}R_j{}^c+k_{cj}R_i{}^c) +2\mathrm{tr}kR_{ji}
+2 g_{ji}R_a{}^ck_c{}^a+(k_{ij}-g_{ji}\mathrm{tr}k)R,
\end{align}
for all $i,j=1,2,3$.
\end{proposition}
\begin{remark}
The operations in \eqref{boxk} are covariant, where the $\partial_t$ differentiations are viewed as applications of the Lie derivative operator $\mathcal{L}_{\partial_t}$, see also \eqref{dtGamma}.
\end{remark}
\begin{remark}
In the gauge $\text{tr}k=0$, many of the terms in \eqref{boxk} can be actually dropped. However, these terms would have to be added later in \eqref{e0Rij4}, when we will verify that a solution to the reduced equations, is actually a solution to the EVE, see Section \ref{sec:verEVE}.
\end{remark}
\begin{proof}
The main ingredient is the derivation of a formula for the time derivative of $R_{ij}$. For this purpose, we introduce some commutation formulas:
\begin{align}\label{comm}
\notag\partial_t\nabla_aX^b_{ij}=&\,\partial_t\bigg[\partial_aX^b_{ij}+\Gamma_{ac}^bX^c_{ij}-\Gamma_{ai}^cX^b_{cj}-\Gamma_{aj}^cX_{ic}^b\bigg]\\
=&\,\nabla_a\partial_tX^b_{ij}+X^c_{ij}\partial_t\Gamma_{ac}^b-X^b_{cj}\partial_t\Gamma^c_{ai}-X^b_{ic}\partial_t\Gamma^c_{ij},
\end{align}
for any $(1,2)$ tensor, where 
\begin{align}\label{dtGamma}
\notag\partial_t\Gamma^b_{ac}=&\,\Phi k^{bl}(\partial_ag_{cl}+\partial_cg_{al}-\partial_lg_{ac})+g^{bl}[\partial_l(\Phi k_{ac})-\partial_a(\Phi k_{cl})-\partial_c(\Phi k_{al})]\\
\notag=&\,2\Phi k^{bl}\Gamma_{ac}^mg_{ml}+g^{bl}[\partial_l(\Phi k_{ac})-\partial_a(\Phi k_{cl})-\partial_c(\Phi k_{al})]\\
=&\,g^{bl}\big[\nabla_l(\Phi k_{ac})-\nabla_a(\Phi k_{cl})-\nabla_c(\Phi k_{al})\big]
\end{align}
Differentiating \eqref{Ricci} and utilising \eqref{comm}, we find
\begin{align}\label{dtRicci}
\partial_tR_{ij}=&\,\nabla_a\partial_t\Gamma_{ji}^a+\Gamma^c_{ji}\partial_t\Gamma^a_{ac}-\Gamma^a_{ci}\partial_t\Gamma^c_{aj}
-\Gamma^a_{jc}\partial_t\Gamma^c_{ai}\\
\notag&-\nabla_j\partial_t\Gamma_{ia}^a-\Gamma^c_{ia}\partial_t\Gamma_{jc}^a+\Gamma^a_{ca}\partial_t\Gamma^c_{ji}+\Gamma^a_{ic}\partial_t\Gamma^c_{ja}\\
\notag&-\partial_t(\Gamma^a_{ca}\Gamma^c_{ji})+\partial_t(\Gamma^a_{ci}\Gamma^c_{aj})\\
\notag=&\,\nabla_a\partial_t\Gamma_{ji}^a-\nabla_j\partial_t\Gamma_{ia}^a
\end{align}
Taking now the time derivative of \eqref{2ndvar} and employing \eqref{dtGamma}-\eqref{dtRicci}, we derive:
\begin{align}\label{boxk2}
\partial_t(\Phi^{-1}\partial_tk_{ij})=&\,\Phi^{-2}\partial_t\Phi\nabla_i\nabla_j\Phi-\Phi^{-1}\nabla_i\nabla_j\partial_t\Phi+\Phi^{-1}\partial_t\Gamma^l_{ij}\partial_l\Phi+\partial_t(k_{ij}\mathrm{tr}k-2{k_i}^lk_{jl})\\
\notag&+\nabla_a \big[\nabla^a(\Phi k_{ij})-\nabla_j(\Phi k_i{}^a)-\nabla_i(\Phi k_j{}^a)\big]\\
\notag&-\nabla_j\big[\nabla^a(\Phi k_{ia})-\nabla_a(\Phi k_i{}^a)-\nabla_i(\Phi k_a{}^a)\big]-\partial_t{\bf R}_{ij}\\
\notag=&\,\Phi^{-2}\partial_t\Phi\nabla_i\nabla_j\Phi-\Phi^{-1}\nabla_i\nabla_j\partial_t\Phi+\Phi^{-1}\partial_t\Gamma^l_{ij}\partial_l\Phi+\partial_t(k_{ij}\mathrm{tr}k-2{k_i}^lk_{jl})\\
\notag&+\Phi\Delta_gk_{ij}+k_{ij}\Delta_g\Phi+2\nabla^a\Phi\nabla_ak_{ij}
-k_i{}^a\nabla_a\nabla_j\Phi-k_j{}^a\nabla_a\nabla_i\Phi\\
\notag&-\nabla^a\Phi(\nabla_jk_{ia}+\nabla_ik_{ja})-\nabla_j\Phi\nabla_ak_i{}^a-\nabla_i\Phi\nabla_ak_j{}^a
-\Phi\nabla_a\nabla_jk_i{}^a-\Phi\nabla_a\nabla_ik_j{}^a\\
\notag&+\mathrm{tr}k\nabla_j\nabla_i\Phi
+\Phi \nabla_j\nabla_i\mathrm{tr}k+\nabla_i\mathrm{tr}k\nabla_j\Phi
+\nabla_j\mathrm{tr}k\nabla_i\Phi-\partial_t{\bf R}_{ij}
\end{align}
Next, we utilise the identity:
\begin{align}\label{nabla2k}
-\Phi\nabla_a\nabla_jk_i{}^a-\Phi\nabla_a\nabla_ik_j{}^a
=-2\Phi R_{aji}{}^ck_c{}^a-R_j{}^ck_{ic}-R_i{}^ck_{jc}-\Phi\nabla_j\nabla_ak_i{}^a-\Phi\nabla_i\nabla_ak_j{}^a,
\end{align}
Note that in 3D the Riemann tensor can be expressed in terms of the Ricci tensor via the identity \cite[(3.2.28)]{Wald}:
\begin{align}\label{Riem=Ric}
R_{aji}{}^c=&\,g_{ai}R_j{}^c-\delta_a{}^cR_{ji}
-g_{ji}R_a{}^c+\delta_j{}^cR_{ai}-\frac{1}{2}R(g_{ai}\delta_j{}^c-\delta_a{}^cg_{ji})\\
\notag\Rightarrow\qquad-2\Phi R_{aji}{}^c{}k_c{}^a=&-2\Phi \bigg[k_{ci}R_j{}^c-\mathrm{tr}kR_{ji}
-g_{ji}R_a{}^ck_c{}^a+k_j{}^aR_{ai}-\frac{1}{2}R(k_{ij}-g_{ji}\mathrm{tr}k)\bigg]
\end{align}
Also, from the contracted Gauss and Codazzi equations we have the identities:
\begin{align}
\label{Gauss}R-|k|^2+(\mathrm{tr}k)^2={\bf R}+2\Phi^{-2}{\bf R}_{tt}\\
\label{Codazzi}
\partial_j\mathrm{tr}k-\nabla^ak_{aj}=\Phi^{-1}{\bf R}_{tj}=\mathcal{G}_j,&&j=1,2,3.
\end{align}
Hence, plugging \eqref{nabla2k},\eqref{Riem=Ric},\eqref{Codazzi} in \eqref{boxk2}, we arrive at the equation:
\begin{align}\label{boxk3}
&\notag e_0^2k_{ij}-\Delta_g k_{ij}\\
=&\,\Phi^{-3}\partial_t\Phi\nabla_i\nabla_j\Phi-\Phi^{-2}\nabla_i\nabla_j\partial_t\Phi+\Phi^{-2}\partial_t\Gamma^l_{ij}\partial_l\Phi+e_0(k_{ij}\mathrm{tr}k-2{k_i}^lk_{jl})\\
\notag&+\Phi^{-1}k_{ij}\Delta_g\Phi
-\Phi^{-1}k_i{}^a\nabla_a\nabla_j\Phi-\Phi^{-1}k_j{}^a\nabla_a\nabla_i\Phi\\
\notag&-\Phi^{-1}\nabla^a\Phi(\nabla_jk_{ia}+\nabla_ik_{ja}-2\nabla_ak_{ij})
+\Phi^{-1}\mathrm{tr}k\nabla_j\nabla_i\Phi
-\nabla_j\nabla_i\mathrm{tr}k\\
\notag&+\Phi^{-1}\nabla_j\Phi(\nabla_i\mathrm{tr}k-\nabla_ak_i{}^a)
+\Phi^{-1}\nabla_i\Phi(\nabla_j\mathrm{tr}k-\nabla_ak_j{}^a)\\
\notag&-3(k_{ci}R_j{}^c+k_{cj}R_i{}^c) +2\mathrm{tr}kR_{ji}
+2 g_{ji}R_a{}^ck_c{}^a+(k_{ij}-g_{ji}\mathrm{tr}k)R\\
\notag&+\nabla_j\mathcal{G}_i+\nabla_i\mathcal{G}_j-e_0{\bf R}_{ij}
\end{align}
This completes the proof  {of the proposition}.
\end{proof}
 {In the case of study, where ${\bf g}$ is a solution to the EVE, under the maximal gauge condition $\text{tr}k=0$, the equation \eqref{e0Rij4} holds trivially and hence so does \eqref{boxk}. Moreover, taking the trace of \eqref{2ndvar} and using \eqref{Gauss}, we obtain the relations:
\begin{align}\label{d/dttrk}
\partial_t\mathrm{tr}k=-\Delta_g\Phi+\Phi [R+(\mathrm{tr}k)^2]-\Phi({\bf R}+{\bf R}_{00})
=-\Delta_g\Phi+|k|^2\Phi+\Phi{\bf R}_{00}
\end{align}
Since $\text{tr}k$ and spacetime Ricci vanish, \eqref{d/dttrk} yields the following elliptic equation for the lapse:
\begin{align}\label{Phieq}
\Delta_g\Phi-|k|^2\Phi=0.
\end{align}
The reduced equations \eqref{1stvar},\eqref{boxk},\eqref{Phieq} form a closed system for $g,k,\Phi$.}
\begin{remark}
A posteriori, having solved the reduced equations, in order to verify the validity of the maximal gauge, we will need to propagate the vanishing of $\mathrm{tr}k$. For this purpose, we compute the trace of \eqref{boxk}, using only \eqref{1stvar},\eqref{Phieq}:
\begin{align}\label{boxtrk}
&\notag e_0^2\mathrm{tr}k-\Delta_g \mathrm{tr}k\\
\notag=&\,e_0(2k^{ij}k_{ij})
+2k^{ij}e_0k_{ij}
+g^{ij}(e_0^2k_{ij}-\Delta_g k_{ij})\\
=&\,e_0(2k^{ij}k_{ij})
+2k^{ij}e_0k_{ij}-g^{ij}e_0(\Phi^{-1}\nabla_i\nabla_j\Phi)
+g^{ij}e_0(k_{ij}\text{tr}k-2k_i{}^lk_{jl})\\
&\notag+2\mathrm{tr}k|k|^2
-2\Phi^{-1}k^{ja}\nabla_a\nabla_j\Phi
+4\Phi^{-1}(\nabla^a\Phi)\mathcal{G}_a\\
\notag=&\,e_0[(\mathrm{tr}k)^2]+4\Phi^{-1}(\nabla^a\Phi)\mathcal{G}_a
\end{align}
\end{remark}
\begin{remark}
To our knowledge, the reduction of the EVE to a wave equation for $k_{ij}$ was first demonstrated in the literature by Choquet-Bruhat--Ruggeri \cite{CBR}. In fact, they derived a system for $P_{ij}:=k_{ij}-g_{ij}\text{tr}k$, using the gauge choice
\begin{align*}
\square_{\bf g}t=0\qquad\Longleftrightarrow\qquad \Phi^{-2}\partial_t\Phi=-\text{tr}k,
\end{align*}
for the $t$-foliation.
\end{remark}
\begin{remark}\label{rem:loss}
The Ricci tensor of $g$ in the RHS of \eqref{boxk} contains terms having two spatial derivatives of $g$. At first glance, this makes the closure of the reduced system more intricate, since \eqref{1stvar} does not gain a derivative in space. However, we demonstrate below, see \eqref{IBPterm}, how to treat these terms in the energy estimates by integrating by parts. Alternatively, these terms could be replaced in the derivation of \eqref{boxk}, in favour of spacetime Ricci, by using the second variation equations \eqref{2ndvar}, involving only $k,\partial_tk,\nabla\nabla\Phi$. In that case, the propagation equation \eqref{e0Rij4} would have to be modified accordingly, adding the appropriate combination of zeroth order Ricci terms.
\end{remark}
\subsection{Boundary data and boundary conditions}\label{subsec:bc}

We assume that the timelike boundary of $\mathcal{M}$, $\mathcal{T}$, is foliated by the compact surfaces $\partial \Sigma_t:=\Sigma_t\cap \mathcal{T}$. Let $H$ denote the induced, $1+2$, Lorentzian metric on the boundary $\mathcal{T}$. For simplicity, we assume that $t\big|_{\mathcal{T}}$ defines a geodesic foliation with respect to the induced metric on the boundary, i.e., $H$ takes the form 
\begin{align}\label{metricH}
H:={\bf g}\big|_{\mathcal{T}}=-[d(t\big|_{\mathcal{T}})]^2+q_t=-[d(t\big|_{\mathcal{T}})]^2+(q_t)_{AB}dx^Adx^B,&&A,B=1,2,
\end{align}
where $x^1,x^2$ are coordinates propagated along the boundary by $\partial_{t|\mathcal{T}}$.\footnote{Defined such that it is parallel to the $H$-gradient of $t\big|_{\mathcal{T}}$, satisfying $\partial_{t|\mathcal{T}}(t\big|_{\mathcal{T}})=1$.}
Combining this assumption with the boundary condition 
\begin{align}\label{Phibdcond}
\Phi=1,\qquad\text{on $\mathcal{T}$},
\end{align}
we infer that the vector field $\partial_t\big|_{\mathcal{T}}$ remains tangent to the boundary $\mathcal{T}$, $\partial_t\big|_{\mathcal{T}}\in T(\mathcal{T})$, and hence, it coincides with $\partial_{t|\mathcal{T}}$. Indeed, from the form of the metrics \eqref{metric},\eqref{metricH} and the definition of the lapse, it follows that the (outward) unit normal to the boundary, $N\perp T(\mathcal{T})$, annihilates $t$:
\begin{align}\label{Nt=0}
N(t)=0,\qquad\text{on $\mathcal{T}$},
\end{align}
which in turn implies that the ${\bf g}$-gradient of $t$ is orthogonal to $N$. 

Moreover, we assume that $\partial\Sigma_0$ has a neighbourhood in $\Sigma_0$, which is covered by the level sets of a defining function $x^3$:\footnote{This can be for example, the Gaussian parameter in a tubular neighbourhood of $\partial\Sigma_0$.}
\begin{align}\label{x3}
x^3=0:\quad\text{on $\partial \Sigma_0$},\qquad x^3<0: \quad\text{in $\Sigma_0\setminus\partial \Sigma_0$}, \qquad dx^3\neq0:\quad\text{on $\partial \Sigma_0$}.
\end{align}
Since $\partial_t\big|_{\mathcal{T}}$ is tangent to the boundary, we may complement $x^3$ with coordinates $x^1,x^2$, near a fixed point $p\in\partial \Sigma_0$, and propagate these along $\partial_t$ to obtain a coordinate system $(t,x^1,x^2,x^3)$ in a spacetime neighbourhood of $p\in\partial \Sigma_0$. Evidently, for small $t$, $x^3$ will remain a defining function of the boundary. Note that by definition, the gradient of $x^3$, $Dx^3$, is normal to the boundary. Hence, setting
\begin{align}\label{dx3}
N:=\frac{Dx^3}{\sqrt{{\bf g}(Dx^3,Dx^3)}},\qquad Dx^3:={\bf g}^{ij}\partial_ix^3\partial_j=g^{3j}\partial_j,\quad g^{33}\partial_3=(g^{33})^\frac{1}{2}N-g^{31}\partial_1-g^{32}\partial_2,
\end{align}
the vector field $N$ is an extension in $\mathcal{M}$, locally around $p$, of the outward unit normal to the boundary $\mathcal{T}$.
\begin{remark}
 {The defining function $x^3$ is global near the boundary, but more than one coordinate patches $x^1,x^2$ have to potentially be used, along the level sets of $x^3$, in order to cover an entire neighbourhood of the boundary $\partial\Sigma_0$. However, for simplicity in the exposition of our overall argument, we will only work with a single patch, projecting the wave equation for $k$ onto this specific frame. Since the wave equation for $k$ \eqref{boxk} is tensorial and since the lapse $\Phi$ is independent of the choice of coordinates on $\Sigma_t$, the whole procedure can then be carried out tensorially in the planes generated by $\partial_1$ and $\partial_2$.}
\end{remark}

{\it Boundary data}: We evaluate $k$ on the boundary against the adapted frame $\partial_A,\partial_B,N$, $A,B=1,2$. The boundary data for the IBVP are given in terms of the 1-parameter family $[q_t]$ of conformal metrics on $\partial\Sigma_t$, see Theorem \ref{mainthm}, which only determine the values of the traceless part of $k$ along the cross sections $\partial\Sigma_t$ with one index raised:
\begin{lemma}\label{lem:bd}
Let $(q_t)_{AB}=\Omega^2[q_t]_{AB}$ and let 
\begin{align}\label{khat}
\hat{k}_A{}^B:=k_A{}^B-\frac{1}{2}\delta_A{}^Bk_C{}^C.
\end{align}
Then, it holds
\begin{align}\label{khatcomp}
\hat{k}_A{}^B=-\frac{1}{2}[q_t]^{BC}\partial_t [q_t]_{AC}+\frac{1}{4}\delta_A{}^B[q_t]^{DC}\partial_t[q_t]_{DC},
\end{align}
for all $A,B=1,2$.
\end{lemma}
\begin{proof}
We compute
\begin{align*}
k_{AB}:=&-\frac{1}{2}\partial_t(\Omega^2[q_t]_{AB}),\qquad k_C{}^C=\Omega^{-2}[q_t]^{DC}k_{DC}=-\frac{1}{2}[q_t]^{DC}\partial_t[q_t]_{DC}-2\Omega^{-1}\partial_t\Omega ,\\
k_A{}^B=&\,\Omega^{-2}[q_t]^{BC}k_{AC}=-\frac{1}{2}\Omega^{-2}[q_t]^{BC}\partial_t(\Omega^2[q_t]_{AC})=-\frac{1}{2}[q_t]^{BC}\partial_t [q_t]_{AC}-\delta_A{}^B\Omega^{-1}\partial_t\Omega.
\end{align*}
Subtracting $\frac{1}{2}\delta_A{}^B$ times the second formula from the third, we notice that the terms involving $\Omega$ cancel out, leaving \eqref{khatcomp}.
\end{proof}
To the rest of the components of $k$, we impose boundary conditions that propagate the maximal gauge and the momentum constraint \eqref{eq:mcons} on the boundary $\mathcal{T}$:
\begin{align}\label{kNNbdcond}
k_{NN}:=&-k_C{}^C, \\
%
\label{kNAbdcond}
\notag\nabla_Nk_{NA}:=&-\nabla_Bk_A{}^B\\
=&-\partial_B(\hat{k}_A{}^B+\frac{1}{2}\delta_A{}^Bk_C{}^C)
+\Gamma_{AB}^C(\hat{k}_C{}^B-\frac{1}{2}\delta_C{}^Bk_{NN})-\Gamma_{BC}^B(\hat{k}_A{}^C-\frac{1}{2}\delta_A{}^Ck_{NN})\\
\notag&+\chi_A{}^Bk_{NB}+\chi_B{}^Bk_{NA}, \\
%
\label{kCCbdcond}
\frac{1}{2}(\nabla_Nk_{NN}-\nabla_Nk_A{}^A):=&-\nabla^Ak_{NA},
\end{align}
where $\chi_{ij}:={\bf g}(D_{\partial_i}\partial_j,N)$ is the second fundamental form of the boundary $\mathcal{T}$, while $\Gamma_{AB}^C$ are Christoffel symbols associated to the induced metric $q_t$ on $\partial\Sigma_t$. 
\begin{remark}
 {The combination of \eqref{kNNbdcond}-\eqref{kNAbdcond}, imply the validity of the momentum constraint \eqref{eq:mcons}, projected on $\partial_1,\partial_2$. However, the last condition \eqref{kCCbdcond} differs slightly from \eqref{eq:mcons}, in the normal direction $N$, since a priori the Neumann type data $N\text{tr}k$ is not known to vanish on the boundary. We found such a modification necessary for the absorption of the boundary terms that arise in the energy estimates for the reduced equation \eqref{boxk}, see Section \ref{sec:enest}. 
Despite this modification, as we show in Section \ref{sec:verEVE}, the above boundary conditions are sufficient for the recovery of the EVE from the reduced equations.}
\end{remark}
\begin{remark}
At first glance, the heavily coupled, mixed Dirichlet-Neumann boundary conditions \eqref{kNNbdcond}-\eqref{kCCbdcond} seem to be losing derivatives in an energy argument for \eqref{boxk}. However, we show that by some careful manipulations, the arising boundary terms can all be absorbed in the main energies and close the estimates, see Proposition \ref{prop:kenest}.
\end{remark}
{\it Initial data}: An initial data set $h,k$ for the EVE on $\Sigma_0$, induces the initial data for \eqref{1stvar} and half of the initial data for \eqref{boxk}. These are sufficient to determine $\Phi$ from \eqref{Phieq}, satisfying the Dirichlet boundary condition \eqref{Phibdcond}. Then the $\partial_tk$ part of the initial data for \eqref{boxk} is fixed by the second variation equations \eqref{2ndvar}, such that the EVE are valid initially on $\Sigma_0$:
\begin{align}\label{dtkinit}
\partial_tk_{ij}\big|_{t=0}= -{\nabla}_i {\nabla}_j\Phi+\Phi( {R}_{ij}+k_{ij}\mathrm{tr}k-2{k_i}^lk_{jl})\big|_{t=0}\qquad\Longleftrightarrow\qquad {\bf R}_{ij}\big|_{\Sigma_0}=0,
\end{align}
for every $i,j=1,2,3$.
Since $g,k$ satisfy the constraints \eqref{eq:hcons}-\eqref{eq:mcons} initially, combining \eqref{dtkinit} with the Gauss and Codazzi equations \eqref{Gauss}-\eqref{Codazzi}, we also have:
\begin{align}\label{R0iinit}
{\bf R}_{00}\big|_{\Sigma_0}={\bf R}_{0i}\big|_{\Sigma_0}=0,\qquad i=1,2,3.
\end{align}
Moreover, $k$  {satisfies} the maximal gauge $\text{tr}k=0$ on $\Sigma_0$. Then, by taking the trace in \eqref{dtkinit}, we arrive at \eqref{d/dttrk} for $t=0$, where by employing \eqref{R0iinit} and the equation \eqref{Phieq} for $\Phi$, we  obtain:
\begin{align}\label{dttrkinit}
\partial_t\text{tr}k\big|_{t=0}=0.
\end{align}
The initial conditions \eqref{dtkinit}-\eqref{dttrkinit} will be used in Section \ref{sec:verEVE} to verify the EVE and the maximal gauge everywhere.

{\it Compatibility conditions}: The initial tensors $h,k$ on $\Sigma_0$ must induce tensors on $S$, which are compatible with the prescribed boundary data. For example, $[h\big|_{\partial\Sigma_0}]=[q_0]$ and $\hat{k}_A{}^B\big|_{S}$ satisfying \eqref{khatcomp}.  The less obvious condition for $\partial_t\hat{k}_A{}^B$ is given through \eqref{dtkinit}:
\begin{align}\label{compcond}
\partial_t\hat{k}_A{}^B\big|_{S}=({R}_A{}^B-\frac{1}{2}\delta_A{}^B\partial_tk_C{}^C-{\nabla}_A {\nabla}^B\Phi)\big|_{S},
\end{align}
where we used the vanishing of $\text{tr}k$ on $\Sigma_0$ and \eqref{Phibdcond}. Notice that the LHS is expressed via \eqref{khatcomp} solely in terms of the conformal metric class $[q_0]$ on $S$ and the time derivatives of $[q_t]$ up to order two, evaluated at $t=0$. Similar relations can be computed to any higher order. Also, note that by virtue of \eqref{Phibdcond}, the Hessian of $\Phi$ above equals:
\begin{align*}
\nabla_A\nabla^B\Phi\big|_S=-\chi_A{}^BN\Phi\big|_S,
\end{align*}
where $N\Phi$ is determined through the Dirichlet to Neumann map for \eqref{Phieq}.

\subsection{The commuted equations and boundary conditions}

We find it suitable to evaluate the wave equation \eqref{boxk} against $\partial_1,\partial_2,N$ and obtain scalarised\footnote{In the case where the initial cross section of the boundary cannot be covered by one coordinate patch, note that all equations below can still be written tensorially in $A,B$.} versions of \eqref{boxk} for the components $k_A{}^B,k_{NA},k_{NN}$.
Recall that $\partial_t$ acts as a Lie derivative. Therefore, we need to take into account the commutation of $\partial_t$ with $N$:
\begin{align}\label{commdtN}
{\bf g}([\partial_t,N],\partial_t)\overset{\eqref{dx3}}{=}0,\qquad [\partial_t,N]^A=(D_{\partial_t}N)^A-(D_N\partial_t)^A=2\Phi K_N{}^A,\qquad[\partial_t,N]^N=\Phi k_{NN},
\end{align}
This implies that
\begin{align}\label{dtkNi}
\notag e_0k_{NA}=&\,e_0(k_{NA})-2 k_N{}^Bk_{BA}- k_{NN}k_{NA}\\
\notag e_0k_{NN}=&\,e_0(k_{NN})-4 k_N{}^Ak_{AN}-2 k_{NN}^2\\
\notag e_0^2k_{NA}=&\,\Phi^{-1}\mathcal{L}_{\partial_t}(\Phi^{-1}\mathcal{L}_{\partial_t}k_{NA})\\
=&\,e_0\big[e_0(k_{NA})-2 k_N{}^Bk_{BA}- k_{NN}k_{NA}\big]
-2k_N{}^Be_0(k_{BA})\\
\notag &-k_{NN}\big[e_0(k_{NA})-2 k_N{}^Bk_{BA}- k_{NN}k_{NA}\big]\\
\notag e_0^2k_{NN}=&\,e_0\big[e_0(k_{NN})-4 k_N{}^Ak_{AN}-2 k_{NN}^2\big]\\
\notag &-4k_N{}^A\big[e_0(k_{NA})-2 k_N{}^Bk_{BA}- k_{NN}k_{NA}\big]\\
\notag &-2k_{NN}\big[e_0(k_{NN})-4 k_N{}^Ak_{AN}-2 k_{NN}^2\big]
\end{align}
Moreover, the Laplacian of $k_{ij}$ expands schematically to:
\begin{align}\label{Deltakij}
\Delta_gk_{ij}=\Delta_g(k_{ij})+\Gamma\star\partial k+\partial\Gamma\star k+\Gamma\star\Gamma\star k,\qquad \Delta_g(k_{ij})=g^{ab}\partial_a\partial_bk_{ij}-g^{ab}\Gamma_{ab}^c\partial_ck_{ij},
\end{align}
yielding second order terms in $g$, in addition to the Ricci terms in the RHS of \eqref{boxk}. 

Thus, we may write \eqref{boxk} in the following form:
\begin{align}\label{boxkij2}
(e_0^2-\Delta_g)k_i{}^j=\mathcal{N}(\Phi_{1;0;2},g_{0;0;2},k_{0;0;1},k_{1;0;0})_i{}^j,&&i,j=A,B,N,
\end{align}
where we use the notation $\mathcal{N}[\{(f_i)_{p_i;r_i;l_i}\}]$ to denote a non-linear expression in the $f_i$'s and their derivatives up to order $p_i+r_i+l_i$ respectively, $p_i$ time derivatives, $r_i$ derivatives among $\partial_1,\partial_2$ and the additional $l_i$ among $N,\partial_1,\partial_2$. If either of $p_i,r_i,l_i$ is less than zero, then the corresponding term is not taken into account. The number of derivatives in each term, summing up the derivatives of each factor,
does not surpass $\max\{p_i+r_i+l_i\}$. 

We commute \eqref{boxkij2} with the tangential vector fields to the boundary: $\partial_t^{r_2}\partial^{r_1}$, $\partial=\partial_1,\partial_2$. The commuted set of equations reads: 
\begin{align}\label{boxkdiff}
\notag(e_0^2-\Delta_g)\partial_t^{r_2}\partial^{r_1}(k_A{}^B)=&\,\mathcal{N}(\Phi_{r_2+1;r_1;2},g_{0;r_1;2},k_{r_2-1;r_1;2},k_{r_2;r_1-1;2},k_{r_2+1;r_1;0})_A^B\\
(e_0^2-\Delta_g)\partial_t^{r_2}\partial^{r_1}(k_{NA})=&\,\mathcal{N}(\Phi_{r_2+1;r_1;2},g_{0;r_1;2},k_{r_2-1;r_1;2},k_{r_2;r_1-1;2},(k)_{r_2+1;r_1;0})_{NA}\\
\notag(e_0^2-\Delta_g)\partial_t^{r_2}\partial^{r_1}(k_{NN})=&\,\mathcal{N}(\Phi_{r_2+1;r_1;2},g_{0;r_1;2},k_{r_2-1;r_1;2},k_{r_2;r_1-1;2},k_{r_2+1;r_1;0})_{NN}
\end{align}
\begin{remark}
The term $g_{0;r_1;2}$ has no time derivatives, since we may replace a time derivative of $g$ in favour of $\Phi,k$, using \eqref{1stvar}.
\end{remark}
Since the boundary data for $\hat{k}_A{}^B=k_A{}^B-\frac{1}{2}\delta_A{}^Bk_C{}^C$ are given, see Lemma \ref{lem:bd}, we may modify $\hat{k}_A{}^B$ such that it has zero Dirichlet boundary data:
\begin{align}\label{tildek}
\tilde{k}_A{}^B=\hat{k}_A{}^B-f_A{}^B,
\end{align}
where $f_A{}^B$ is a smooth extension in $\mathcal{M}$ of $\hat{k}_A{}^B\big|_{\mathcal{T}}$. Then the system \eqref{boxkdiff} becomes:
\begin{align}\label{boxkdiff2}
\notag(e_0^2-\Delta_g)\partial_t^{r_2}\partial^{r_1}(\tilde{k}_A{}^B)=&\,\mathcal{N}(\Phi_{r_2+1;r_1;2},g_{0;r_1;2},k_{r_2-1;r_1;2},k_{r_2;r_1-1;2},
k_{r_2+1;r_1;0},f_{r_2+2;r_1;0},f_{r_2;r_1;2})_A^B\\
(e_0^2-\Delta_g)\partial_t^{r_2}\partial^{r_1}(k_C{}^C)=&\,\mathcal{N}(\Phi_{r_2+1;r_1;2},g_{0;r_1;2},k_{r_2-1;r_1;2},k_{r_2;r_1-1;2},
k_{r_2+1;r_1;0},f_{r_2+2;r_1;0},f_{r_2;r_1;2})_C^C\\
\notag(e_0^2-\Delta_g)\partial_t^{r_2}\partial^{r_1}(k_{NA})=&\,\mathcal{N}(\Phi_{r_2+1;r_1;2},g_{0;r_1;2},k_{r_2-1;r_1;2},k_{r_2;r_1-1;2},
k_{r_2+1;r_1;0},f_{r_2+2;r_1;0},f_{r_2;r_1;2})_{NA}\\
\notag(e_0^2-\Delta_g)\partial_t^{r_2}\partial^{r_1}(k_{NN})=&\,\mathcal{N}(\Phi_{r_2+1;r_1;2},g_{0;r_1;2},k_{r_2-1;r_1;2},k_{r_2;r_1-1;2},
k_{r_2+1;r_1;0},f_{r_2+2;r_1;0},f_{r_2;r_1;2})_{NN}
\end{align} 
where $k$ in the RHS includes also $\tilde{k}_A{}^B$.

The corresponding differentiated equations for $g_{ij}$ and $\Phi$ read:
\begin{align}
\label{1stvardiff}\partial_tN^l\partial^rg_{ij}=&\,-2N^l\partial^r(\Phi k_{ij})+\mathcal{N}(k_{0;0;l-1},\Phi_{0;0;l-1},g_{0;r;l})_{ij}\\
\label{DeltaPhidiff}h^{AB}\partial_A\partial_B\partial_t^{r_2}\partial^{r_1}\Phi+NN\partial_t^{r_2}\partial^{r_1}\Phi=&\,\partial_t^{r_2}\partial^{r_1}(|k|^2\Phi)+\mathcal{N}(\Phi_{r_2-1;r_1;2},\Phi_{r_2;r_1-1;2},g_{0;r_1;1},k_{r_2-1;r_1;1})
\end{align}
for every $i,j=1,2,3$.
\begin{remark}
The top order terms in the RHS of \eqref{boxkdiff2}, containing $r_1+2$ spatial derivatives of $g$, cannot be directly estimated in $L^2$ in terms of the energy of the wave operator in the LHS, since \eqref{1stvardiff} does not gain a derivative in space. We show how to treat these terms in the energy estimates, using the structure of the equations, in the proof of Proposition \ref{prop:kenest}.
\end{remark}
We will also use the boundary conditions \eqref{kNAbdcond}-\eqref{kCCbdcond}, commuted with $\partial_t^{r_2}\partial^{r_1}$:
\begin{align}
\label{kNAbdconddiff}N\partial_t^{r_2}\partial^{r_1}k_{NA}
=-\partial_B\partial_t^{r_2}\partial^{r_1}f_A{}^B-\frac{1}{2}\partial_A\partial_t^{r_2}\partial^{r_1}k_C{}^C
+\mathcal{N}(k_{r_2;r_1-1;1},k_{r_2-1;r_1;1},g_{0;r_1;1},\Phi_{r_2;r_1;1})\\
\label{kCCbdconddiff}N\partial_t^{r_2}\partial^{r_1}k_{NN}-N\partial_t^{r_2}\partial^{r_1}k_A{}^A=-2\partial^A\partial_t^{r_2}\partial^{r_1}k_{NA}+\mathcal{N}(k_{r_2;r_1-1;1},k_{r_2-1;r_1;1},g_{0;r_1;1},\Phi_{r_2;r_1;1}),
\end{align}
\begin{remark} 
The boundary term containing $g_{0;r_1;1}$ cannot be directly absorbed in $L^2$ by the energy of the system \eqref{boxkdiff2}, via a trace inequality, at top order $r_2=0$. However, we may use the fact that $g$ gains a derivative in $\partial_t$ to make a trade off and close the energy estimates for $k$, see the proof of Proposition \ref{prop:kenest}.
\end{remark}

\section{Local existence}\label{sec:enest}

Our main goal in this section is to show how to derive energy estimates for the system \eqref{1stvar},\eqref{boxk},\eqref{Phieq}, subject to the boundary conditions \eqref{Phibdcond},\eqref{kNNbdcond},\eqref{kNAbdcond},\eqref{kCCbdcond}. In the end of this section we outline the steps that upgrade these energy estimates to a Picard iteration argument, hence, proving local existence for the reduced system of equations in the same energy spaces. The fact that the Ricci tensor of a solution to the reduced system vanishes is then demonstrated in the next section, which completes the proof of Theorem \ref{mainthm}. 

 {First, we argue that the problem can be localised in a neighbourhood of the boundary by realizing the following three steps:
 
(P1) Consider the solution ${\bf g}_1$ to the EVE in the domain of dependence $D(\Sigma_0)$ of the initial hypersurface $\Sigma_0$. We may consider a timelike hypersurface $\mathcal{T}_{ind}:=\{x^3=\varepsilon\}$, for some $\varepsilon>0$.
\begin{figure}[h!]
  \centering
    \includegraphics[width=0.4\textwidth]{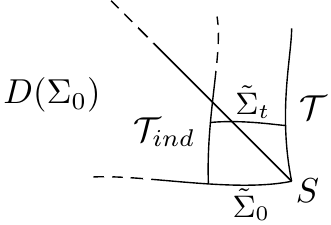}
      \caption{The domain of dependence $\mathcal{D}(\Sigma_0)$.}
\label{Region1}
\end{figure}

(P2) Then, we restrict our attention to the region bounded between $\mathcal{T}_{ind},\mathcal{T}$. In particular, we solve the reduced equations \eqref{1stvar}, \eqref{boxk}, \eqref{Phieq}, by imposing the boundary conditions \eqref{Phibdcond}, \eqref{kNNbdcond}, \eqref{kNAbdcond}, \eqref{kCCbdcond} on the artificial timelike boundary $\mathcal{T}_{ind}$, as well as considering  any\footnote{ {We could also consider the induced data on $\mathcal{T}_{ind}$ from the solution in the domain of dependence region $D(\Sigma_0)$. However, it makes no difference for our argument, since we discard part of the solution to the reduced equations near the artificial boundary.}} regular Dirichlet boundary data for $\hat{k}_A{}^B$,
everything defined with respect to a maximal foliation $\tilde{\Sigma}_t$, as depicted in Figure \ref{Region2}. 
\begin{figure}[h!]
  \centering
    \includegraphics[width=0.4\textwidth]{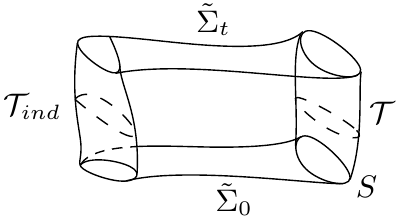}
      \caption{The region between $\mathcal{T}_{ind},\mathcal{T}$.}
\label{Region2}
\end{figure}

(P3) After having solved the above reduced system of equations, in the  region between $\mathcal{T}_{ind},\mathcal{T}$, and have concluded that it consitutes a solution to the EVE, ${\bf g}_2$, see Section \ref{sec:verEVE}, we then define our final vacuum Lorentzial manifold by considering the metric
\begin{align}\label{finalg}
{\bf g}=\left\{\begin{array}{ll}
{\bf g}_1,&D(\Sigma_0)\\
{\bf g}_2,&D(\tilde{\Sigma}_0)\cup D_{free}
\end{array}\right.,
\end{align}
derived from the two solutions ${\bf g}_1,{\bf g}_2$ in the union of the three regions depicted in Figure \ref{Region3}. The fact that ${\bf g}$ is well-defined follows from the classical geometric uniqueness for the initial value problem in the domain dependence of $\tilde{\Sigma}_0$, which implies that ${\bf g}_1,{\bf g_2}$ are isometric in $D(\tilde{\Sigma}_0)$. 
\begin{figure}[h!]
  \centering
    \includegraphics[width=0.3\textwidth]{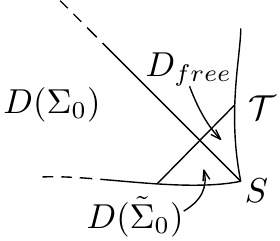}
      \caption{The domain $D(\Sigma_0)\cup D_{free}$ of the resulting solution.}
\label{Region3}
\end{figure}

\noindent
The domain of ${\bf g}$ obviously covers a future spacetime  neighbourhood of $\Sigma_0$. This completes our localisation procedure.

In the rest of this section, we treat the second part (P2), solving the reduced equations in the cylindrical region between $\mathcal{T}_{ind},\mathcal{T}$.
}

Suppose that the boundary values of $\hat{k}_A{}^B$ on $\mathcal{T}_{ind}$ have been incorporated in the definition \eqref{tildek} of $\tilde{k}_A{}^B$, such that $\tilde{k}_A{}^B$ has homogeneous Dirichlet boundary data on $\mathcal{T}_{ind}$ as well. \\
We will be working with the following energies: 
\begin{align}
\label{Etotal}
E_{total}(t):=&\,E_k(t)+\sum_{i,j=1}^3\|g_{ij}\|^2_{H^{r+1}(\tilde{\Sigma}_t)}+\|\Phi\|_{H^{r+2}(\tilde{\Sigma}_t)}^2+\sum_{i=0}^r\|\partial_t^{i+1}\Phi\|_{H^{r+1-i}(\tilde{\Sigma}_t)}^2,\\
\notag
E_k(t):=&\sum_{r_1+r_2\leq r}\int_{\tilde{\Sigma}_t}\bigg(\sum_{A,B}\big[(e_0\partial_t^{r_2}\partial^{r_1}\tilde{k}_A{}^B)^2+|\nabla\partial_t^{r_2}\partial^{r_1}\tilde{k}_A{}^B|^2_g\big]+(e_0\partial_t^{r_2}\partial^{r_1}k_C{}^C)^2\\
&+|\nabla\partial_t^{r_2}\partial^{r_1}k_C{}^C|^2_g
\label{Ek}+4h^{AB}\big[e_0\partial_t^{r_2}\partial^{r_1}k_{NA}e_0\partial_t^{r_2}\partial^{r_1}k_{NB}
+\partial^i\partial_t^{r_2}\partial^{r_1}k_{NA}\partial_i\partial_t^{r_2}\partial^{r_1}k_{NB}\big]\\
&+(e_0\partial_t^{r_2}\partial^{r_1}k_{NN})^2+|\nabla\partial_t^{r_2}\partial^{r_1}k_{NN}|^2_g\bigg)\mathrm{vol}_{\tilde{\Sigma}_t},
\notag
\end{align}
where $|\nabla u|_g^2=g^{ij}\partial_iu\partial_ju$, $\mathrm{vol}_{\tilde{\Sigma}_t}$ is the intrinsic volume form, and 
\begin{align}\label{Hr}
\|u\|_{H^r(\tilde{\Sigma}_t)}^2:=\sum_{r_1+r_2\leq r}\int_{\tilde{\Sigma}_t}(N^{r_2}\partial^{r_1}u)^2\mathrm{vol}_{\tilde{\Sigma}_t}.
\end{align}
 {As part of our assumption, $g$ is a Riemannian metric initially. In what follows, we fix an $r\ge3$.}

At this point we make the assumption that a solution to the reduced system  \eqref{1stvar},\eqref{boxk},\eqref{Phieq} exists, it lives in the above energy space and it satisfies
\begin{align}\label{bootstrap}
E_{total}(t)\leq C_0,&&t\in[0,T_0],
\end{align}
for some $C_0,T_0>0$. We will then show that \eqref{bootstrap} implies the estimate 
\begin{align}\label{Etotalest}
E_{total}(t)\leq C [E_{total}(0)+C_{b.d.}],&&t\in[0,T],
\end{align}
for some $T\leq T_0$ sufficiently small, depending on $C_0$, where $C>0$ is a uniform constant independent of $C_0$, and $C_{b.d.}>0$ depends continuously on the $L^\infty\big([0,T];H^{r+2-i}(\tilde{\Sigma}_t)\big)$ norm of $\partial_t^if_A{}^B$, $i=0,\ldots,r+2$, $A,B=1,2$. 
Hence, by choosing $C_0=C [E_{total}(0)+C_{b.d.}]$ and $T_0=T$ in the first place, our argument can be transformed into an iteration scheme, see the discussion in the end of the section.
\begin{lemma}\label{lem:Econtrol}
The energy $E_{total}(t)$ controls the corresponding energies for $k$ that include $N$ derivatives: 
\begin{align}\label{Econtrol}
\sup_{\tau\in[0,t]}\int_{\tilde{\Sigma}_\tau}(e_0\partial^{r_3}_\tau N^{r_2}\partial^{r_1}u)^2+|\nabla \partial_\tau^{r_3}N^{r_2}\partial^{r_1}u|^2_g\mathrm{vol}_{\tilde{\Sigma}_\tau}\lesssim \sup_{\tau\in[0,t]}E_{total}(\tau)+C_{b.d.},
\end{align}
for all $r_3+r_2+r_1\leq r$, $u=\tilde{k}_A{}^B,k_C{}^C,k_{NA},k_{NN}$, $A,B=1,2$, the implicit constant depending on $C_0$. 
\end{lemma}
\begin{proof}
From the wave equations \eqref{boxkdiff2} we have 
\begin{align}\label{NNkij}
\begin{split}
NN\partial^{r_3}_t\partial^{r_1}u=&\,\mathcal{N}(\Phi_{r_3+1;r_1;2},g_{0;r_1;2},u_{r_3-1;r_1;2},u_{r_3;r_1-1;2},
u_{r_3+1;r_1;0},f_{r_3+2;r_1;0},f_{r_3;r_1;2})\\
&+e_0^2\partial^{r_3}_t\partial^{r_1}u-h^{AB}\partial_A\partial_B\partial^{r_3}_t\partial^{r_1}u,
\end{split}
\end{align}
where $u$ in $\mathcal{N}$ can be any from $\tilde{k}_A{}^B,k_C{}^C,k_{NA},k_{NN}$.

Thus, by iteratively differentiating \eqref{NNkij} with $N^{r_2}$ and taking the $L^2$ norms of both sides, we derive \eqref{Econtrol} by finite induction in $r_2$. The non-linear terms can be estimated in the standard way by making use of the classical Sobolev inequality $\|u\|_{L^\infty(\tilde{\Sigma}_t)}\lesssim \|u\|_{H^2(\tilde{\Sigma}_t)}$ and \eqref{bootstrap}. As for the term $g_{0;r_1;2}$, at top order $r_1=r$, which has one more derivative compared to the norm in the definition \eqref{Etotal} of the energy $E_{total}(t)$, we may replace it by its initial value and the time integral of the RHS of \eqref{1stvardiff}. The arising top order terms in $k$ have then a smallness in $t$, after taking the $\sup_{\tau\in[0,t]}$ and can therefore be absorbed in the LHS.
\end{proof}
\begin{lemma}\label{lem:gijPhiest}
The $H^s$ norms of $g_{ij},\Phi$ are controlled by the energy of $k_i{}^j$ in the following fashion:
\begin{align}\label{gijPhiest}
\notag
\sum_{i,j=1}^3\big\|g_{ij}-g_{ij}\big|_{t=0}\big\|_{H^{r+1}}^2+\big\|\Phi-\Phi\big|_{t=0}\big\|_{H^{r+2}}^2+\sum_{l=0}^{r-1}\big\|\partial_t^{l+1}\Phi-\partial_t^{l+1}\Phi\big|_{t=0}\big\|_{H^{r+2-l}}^2\\
\lesssim t \sup_{\tau\in[0,t]}E_k(\tau)+tC_{b.d.},\\
\|\partial_t^{r+1}\Phi\|_{H^2}^2\lesssim E_k(t)+C_{b.d.},
\notag
\end{align}                                                                                                                                                                                                                                                                                                                                                                                                                                                                                                                                                                                                                                                                                                                                                                                                                                                                                                                                                                                                                                                                                                                                            
for all $t\in[0,T]$, the implicit constant depending on $C_0$ only in the first inequality, provided $T>0$ is sufficiently small.
\end{lemma}
\begin{proof}
The part of the estimate involving $g_{ij}$ follows immediately by integrating \eqref{1stvardiff} in $[0,t]$, taking the $L^2$ norms of both sides and applying Gronwall's inequality. To estimate $\Phi$ and its derivatives that appear in \eqref{1stvardiff}, we utilise \eqref{bootstrap}.

On the other hand, from \eqref{DeltaPhidiff} and standard elliptic estimates\footnote{Note that the original elliptic operator acting on $\Phi-1$, cf. \eqref{Phieq}, has trivial kernel.} for the difference $\Phi-1$ with homogeneous Dirichlet boundary data \eqref{Phibdcond}, treating the factors $g_{0;r_1;1}$ as coefficients by virtue of \eqref{bootstrap}, we have
\begin{align}\label{Phiest}
\sum_{r_1+r_2\leq r}\|\partial_t^{r_2}\partial^{r_1}\Phi\|_{H^2}^2+\|\partial_t^{r+1}\Phi\|^2_{H^2}\lesssim E_k(t)+C_{b.d.}
\end{align}
We may also differentiate the equation \eqref{DeltaPhidiff} with $N^{r_3}$ and obtain the estimate
\begin{align}\label{Phiest2}
\|\Phi\|_{H^{r+2}}^2+\sum_{l\leq r}\|\partial_t^{l+1}\Phi\|_{H^{r+2-l}}^2\lesssim E_k(t)+C_{b.d.},
\end{align}
by induction in $r_3$, cf. Lemma \ref{lem:Econtrol}.

The improved bounds now, with smallness in $t$, for the lapse terms $\partial_t^l\Phi$, $l\leq r$, follow by integrating $\partial_t^{l+1}\Phi$ in $[0,t]$ and using the elliptic estimate \eqref{Phiest2} for the integrand. These in turn can be used to rederive \eqref{Phiest2} with an implicit constant that does not depend on $C_0$. In fact, all the less than top order terms in $g,k,\Phi$ that are viewed as coefficients can be estimated by their initial values plus $tC_0$, using \eqref{bootstrap}.
\end{proof}
Now we can proceed to the proof of the energy estimate \eqref{Etotalest}. 
According to the previous lemma, the part of $E_{total}(t)$ that corresponds to $g,\Phi$ can be essentially controlled by $E_k(t)$. Hence, the overall estimate \eqref{Etotalest} reduces to the corresponding one for $E_k(t)$.
\begin{proposition}\label{prop:kenest}
The energy $E_k(t)$ satisfies the estimate:
\begin{align}\label{Ekest}
E_k(t)\leq C[E_{total}(0)+C_{b.d.}],&&t\in[0,T],
\end{align}
for a $T>0$ sufficiently small and a constant $C>0$ independent of $C_0$.
\end{proposition}
Combining \eqref{gijPhiest},\eqref{Ekest}, we obtain \eqref{Etotalest} for $T$ sufficiently small. 
\begin{proof}
The standard energy argument for the wave equations \eqref{boxkdiff2}, making use of Lemma \ref{lem:gijPhiest} and \eqref{bootstrap}, yields the energy inequality:
\begin{align}\label{enineq}
\notag
&\frac{1}{2}\partial_tE_k(t)
-\sum_{r_1+r_2\leq r} \int_{\partial\tilde{\Sigma}_t}\bigg(\partial_t \partial^{r_2}_t\partial^{r_1}k_C{}^C\nabla_N\partial^{r_2}_t\partial^{r_1}k_C{}^C
+4h^{AB}\partial_t \partial^{r_2}_t\partial^{r_1}k_{NA}\nabla_N\partial^{r_2}_t\partial^{r_1}k_{NB}\\
&+\partial_t \partial^{r_2}_t\partial^{r_1}k_{NN}\nabla_N\partial^{r_2}_t\partial^{r_1}k_{NN}\bigg)\mathrm{vol}_{\partial\tilde{\Sigma}_t}dt+\int_{\tilde{\Sigma}_t}\overline{\partial}^2\partial^r g\star \partial_t\partial^rk\mathrm{vol}_{\tilde{\Sigma}_t}\\
\lesssim &\,[E_k(t)+C_{b.d.}]
\notag
\end{align}
where the last term in the preceding LHS contains an excessive number of spatial derivatives of $g$ that comes from $g_{0;r_1;2}$, $r_1=r$, in the RHS of \eqref{boxkdiff2}, $\overline{\partial}=\partial_1,\partial_2,\partial_3$. In fact, this type of product contains other lower order factors, whose first derivatives are in $L^\infty$, and which we choose to suppress here for simplicity, since they do not matter. 
\begin{align}\label{IBPterm}
\notag&\int_{\tilde{\Sigma}_t}\overline{\partial}^2\partial^r g\star \partial_t\partial^rk\mathrm{vol}_{\tilde{\Sigma}_t}=-\int_{\tilde{\Sigma}_t}\overline{\partial}^2\partial^{r-1} g\star \partial_t\partial^{r+1}k+\overline{\partial}^2\partial^{r-1} g\star \partial_t\partial^rk\frac{\partial\mathrm{vol}_{\tilde{\Sigma}_t}}{\mathrm{vol}_{\tilde{\Sigma}_t}}\mathrm{vol}_{\tilde{\Sigma}_t}\\
=&-\partial_t\bigg(\int_{\tilde{\Sigma}_t}\overline{\partial}^2\partial^{r-1} g\star\partial^{r+1}k\mathrm{vol}_{\tilde{\Sigma}_t}\bigg)+\int_{\tilde{\Sigma}_t}\overline{\partial}^2\partial^{r-1} k\star \partial^{r+1}k+\overline{\partial}^2\partial^{r-1} g\star \partial_t\partial^rk\frac{\partial\mathrm{vol}_{\tilde{\Sigma}_t}}{\mathrm{vol}_{\tilde{\Sigma}_t}}\mathrm{vol}_{\tilde{\Sigma}_t}\\
\tag{by \eqref{gijPhiest}, $T$ sufficiently small}\leq&-\partial_t\bigg(\int_{\tilde{\Sigma}_t}\overline{\partial}^2\partial^{r-1} g\star \partial^{r+1}k\mathrm{vol}_{\tilde{\Sigma}_t}\bigg)+C[E_k(t)+C_{b.d.}]\notag
\end{align}
We compute the integrands in the boundary terms using 
conditions \eqref{kNNbdcond},\eqref{kNAbdconddiff},\eqref{kCCbdconddiff}:
\begin{align}\label{enkbdterms}
\notag
&\partial_t \partial^{r_2}_t\partial^{r_1}k_C{}^CN\partial^{r_2}_t\partial^{r_1}k_C{}^C+\sum_A4h^{AB}\partial_t \partial^{r_2}_t\partial^{r_1}k_{NA}N\partial^{r_2}_t\partial^{r_1}k_{NB}+\partial_t\partial^{r_2}_t\partial^{r_1}k_{NN}N\partial^{r_2}_t\partial^{r_1}k_{NN}\\
\notag=&\,\partial_t\partial^{r_2}_t\partial^{r_1}k_C{}^C(N\partial^{r_2}_t\partial^{r_1}k_C{}^C-N\partial^{r_2}_t\partial^{r_1}k_{NN})
-\sum_A4h^{AB}\partial_t \partial^{r_2}_t\partial^{r_1}k_{NA}\bigg[\partial_C\partial_t^{r_2}\partial^{r_1}f_B{}^C\\
&+\frac{1}{2}\partial_B\partial_t^{r_2}\partial^{r_1}k_C{}^C+\mathcal{N}(k_{r_2;r_1;0},k_{r_2-1;r_1;1},g_{0;r_1;1},\Phi_{r_2;r_1;1})\bigg]\\
\notag=&\,2\partial_t\partial^{r_2}_t\partial^{r_1}k_C{}^C\partial^A\partial^{r_2}_t\partial^{r_1}k_{NA}-2h^{AB}\partial_t\partial^{r_2}_t\partial^{r_1}k_{NA}\partial_B\partial^{r_2}_t\partial^{r_1}k_C{}^C\\
\notag&-4h^{AB}\partial_t \partial^{r_2}_t\partial^{r_1}k_{NA}\bigg[\partial_C\partial^{r_2}_t\partial^{r_1}f_B{}^C
+\mathcal{N}(k_{r_2;r_1;0},k_{r_2-1;r_1;1},g_{0;r_1;1},\Phi_{r_2;r_1;1})\bigg]\\
&+\partial_t\partial^{r_2}_t\partial^{r_1}k_C{}^C\mathcal{N}(k_{r_2;r_1;0},k_{r_2-1;r_1;1},g_{0;r_1;1},\Phi_{r_2;r_1;1})
\notag
\end{align}
We plug \eqref{enkbdterms} into the boundary terms in \eqref{enineq} and compute: 
\begin{align}\label{bdtermsest}
\notag
& \int_{\partial\tilde{\Sigma}_t}2\partial_t\partial^{r_2}_t\partial^{r_1}k_C{}^C\partial^A\partial^{r_2}_t\partial^{r_1}k_{NA}\mathrm{vol}_{\partial\tilde{\Sigma}_t} \\
=&- \int_{\partial\tilde{\Sigma}_t}2h^{AB}\partial_t\partial_A\partial^{r_2}_t\partial^{r_1}k_C{}^C\partial^{r_2}_t\partial^{r_1}k_{NB}+\Gamma\star\partial^{r_2}_t\partial^{r_1}k\star\partial^{r_2}_t\partial^{r_1}k \mathrm{vol}_{\partial\tilde{\Sigma}_t} \\
\notag=&-\partial_t\bigg(\int_{\partial\tilde{\Sigma}_t}2\partial^A\partial^{r_2}_t\partial^{r_1}k_C{}^C\partial^{r_2}_t\partial^{r_1}k_{NA} \mathrm{vol}_{\partial\tilde{\Sigma}_t}\bigg)+ \int_{\partial\tilde{\Sigma}_t}2\partial^A\partial^{r_2}_t\partial^{r_1}k_C{}^C\partial_t\partial^{r_2}_t\partial^{r_1}k_{NA}\mathrm{vol}_{\partial\tilde{\Sigma}_t}\\
\notag&+\int_{\partial\tilde{\Sigma}_t}k\star\partial^A\partial^{r_2}_t\partial^{r_1}k_C{}^C\partial^{r_2}_t\partial^{r_1}k_{NA}-\Gamma\star\partial^{r_2}_t\partial^{r_1}k\star\partial^{r_2}_t\partial^{r_1}k  \mathrm{vol}_{\partial\tilde{\Sigma}_t}
\end{align}
where
\begin{align}\label{bdtermest2}
\notag
&\int_{\partial\tilde{\Sigma}_t}k\star\partial^A\partial^{r_2}_t\partial^{r_1}k_C{}^C\partial^{r_2}_t\partial^{r_1}k_{NA}-\Gamma\star\partial^{r_2}_t\partial^{r_1}k\star\partial^{r_2}_t\partial^{r_1}k  \mathrm{vol}_{\partial\tilde{\Sigma}_t}\\
=&\int_{\tilde{\Sigma}_t}N\bigg(k\star\partial^A\partial^{r_2}_t\partial^{r_1}k_C{}^C\partial^{r_2}_t\partial^{r_1}k_{NA}-\Gamma\star\partial^{r_2}_t\partial^{r_1}k\star\partial^{r_2}_t\partial^{r_1}k \bigg) \mathrm{vol}_{\tilde{\Sigma}_t}\\
\lesssim&\, E_k(t)+C_{b.d.}\notag
\end{align}
 {In \eqref{bdtermest2}, when $N$ hits the factor $\partial^A\partial^{r_2}_t\partial^{r_1}k_C{}^C$, we commute $N,\partial_A$ and integrate by parts in $\partial_A$ to obtain the last bound.}

The terms in the last two lines of \eqref{enkbdterms} are treated in the same manner, via a trace inequality, exploiting the fact that in the higher order terms, there is always a $\partial_t,\partial_1,\partial_2$ derivative that can be integrated by parts, without creating new boundary terms.
\begin{align}\label{bdtermsest3}
\notag&\int_{\partial\tilde{\Sigma}_t}\partial_t \partial^{r_2}_t\partial^{r_1}k\star \bigg[\partial\partial^{r_2}_t\partial^{r_1}f+\mathcal{N}(k_{r_2;r_1;0},k_{r_2-1;r_1;1},g_{0;r_1;1},\Phi_{r_2;r_1;1})\bigg]\mathrm{vol}_{\partial\tilde{\Sigma}_t}\\
=&\,\partial_t\bigg(\int_{\partial\tilde{\Sigma}_t} \partial^{r_2}_t\partial^{r_1}k\star \partial\partial^{r_2}_t\partial^{r_1}f\mathrm{vol}_{\partial\tilde{\Sigma}_t}\bigg)-\int_{\partial\tilde{\Sigma}_t} \partial^{r_2}_t\partial^{r_1}k\star \frac{\partial_t( \partial\partial^{r_2}_t\partial^{r_1}f\mathrm{vol}_{\partial\tilde{\Sigma}_t})}{\mathrm{vol}_{\tilde{\Sigma}_t}}\mathrm{vol}_{\tilde{\Sigma}_t}\\
\notag&+\int_{\tilde{\Sigma}_t}N\bigg[\partial_t \partial^{r_2}_t\partial^{r_1}k\star \mathcal{N}(k_{r_2;r_1;0},k_{r_2-1;r_1;1},g_{0;r_1;1},\Phi_{r_2;r_1;1})\bigg]\mathrm{vol}_{\tilde{\Sigma}_t}\\
\notag\leq&\,\partial_t\bigg(\int_{\partial\tilde{\Sigma}_t} \partial^{r_2}_t\partial^{r_1}k\star \partial\partial^{r_2}_t\partial^{r_1}f\mathrm{vol}_{\partial\tilde{\Sigma}_t}\bigg)+CC_0[E_k(t)+C_{b.d.}]\\
\notag&+\int_{\tilde{\Sigma}_t}\partial_t N\partial^{r_2}_t\partial^{r_1}k\star \mathcal{N}(k_{r_2;r_1;0},k_{r_2-1;r_1;1},g_{0;r_1;1},\Phi_{r_2;r_1;1})\mathrm{vol}_{\tilde{\Sigma}_t}\\
\notag&+\int_{\tilde{\Sigma}_t}\partial_t \partial^{r_2}_t\partial^{r_1}k\star \mathcal{N}(k_{r_2;r_1;0},k_{r_2-1;r_1;1},g_{0;r_1;1},\Phi_{r_2;r_1;1})\mathrm{vol}_{\tilde{\Sigma}_t}\\
\notag\leq&\,\partial_t\bigg(\int_{\partial\tilde{\Sigma}_t} \partial^{r_2}_t\partial^{r_1}k\star \partial\partial^{r_2}_t\partial^{r_1}f\mathrm{vol}_{\partial\tilde{\Sigma}_t}\bigg)+CC_0[E_k(t)+C_{b.d.}]\\
\notag&+\partial_t\bigg(\int_{\tilde{\Sigma}_t} N\partial^{r_2}_t\partial^{r_1}k\star \mathcal{N}(k_{r_2;r_1;0},k_{r_2-1;r_1;1},g_{0;r_1;1},\Phi_{r_2;r_1;1})\mathrm{vol}_{\tilde{\Sigma}_t}\bigg)\\
\notag&-\int_{\tilde{\Sigma}_t} N\partial^{r_2}_t\partial^{r_1}k\star \mathcal{N}(k_{r_2;r_1;0},k_{r_2-1;r_1;1},g_{0;r_1;1},\Phi_{r_2;r_1;1})\mathrm{vol}_{\tilde{\Sigma}_t}\\
\notag&+\int_{\tilde{\Sigma}_t}\partial_t \partial^{r_2}_t\partial^{r_1}k\star \mathcal{N}(k_{r_2-1;r_1;2})\mathrm{vol}_{\tilde{\Sigma}_t}\\
\notag\leq&\,\partial_t\bigg(\int_{\partial\tilde{\Sigma}_t} \partial^{r_2}_t\partial^{r_1}k\star \partial\partial^{r_2}_t\partial^{r_1}f\mathrm{vol}_{\partial\tilde{\Sigma}_t}\bigg)+CC_0[\sup_{\tau\in[0,t]}E_k(\tau)+C_{b.d.}]\qquad\qquad\quad\text{(IBP in $\partial$ and use of \eqref{Econtrol},\eqref{gijPhiest})}\\
\notag&+\partial_t\bigg(\int_{\tilde{\Sigma}_t} N\partial^{r_2}_t\partial^{r_1}k\star \mathcal{N}(k_{r_2;r_1;0},k_{r_2-1;r_1;1},g_{0;r_1;1},\Phi_{r_2;r_1;1})\mathrm{vol}_{\tilde{\Sigma}_t}\bigg)
\end{align}
Combining \eqref{enineq}-\eqref{bdtermsest3} and integrating in $[0,t]$, we obtain the inequality: 
\begin{align}\label{enineq2}
\notag
\frac{1}{2}E_k(t)\leq&\,C[E_{total}(0)+C_{b.d.}]+CC_0[t\sup_{\tau\in[0,t]}E_k(\tau) +tC_{b.d.}]
-\int_{\tilde{\Sigma}_t}\overline{\partial}^2\partial^{r-1} g\star\partial^{r+1}k\mathrm{vol}_{\tilde{\Sigma}_t}\\
\notag&
+\sum_{r_1+r_2\leq r}\bigg[\int_{\tilde{\Sigma}_t} N\partial^{r_2}_t\partial^{r_1}k\star \mathcal{N}(k_{r_2;r_1;0},k_{r_2-1;r_1;1},g_{0;r_1;1},\Phi_{r_2;r_1;1})\mathrm{vol}_{\tilde{\Sigma}_t}\\
&+\int_{\partial\tilde{\Sigma}_t} \partial^{r_2}_t\partial^{r_1}k\star \partial\partial^{r_2}_t\partial^{r_1}f\mathrm{vol}_{\partial\tilde{\Sigma}_t}
-\int_{\partial\tilde{\Sigma}_t}2h^{AB}\partial_A\partial^{r_2}_t\partial^{r_1}k_C{}^C\partial^{r_2}_t\partial^{r_1}k_{NB}\mathrm{vol}_{\partial\tilde{\Sigma}_t}\bigg]\\
\notag\leq&\,CE_{total}(0)+\frac{C}{\varepsilon}C_{b.d.}+CC_0\big[\frac{t}{\varepsilon}\sup_{\tau\in[0,t]}E_k(\tau) +tC_{b.d.}\big]+\varepsilon E_k(t)\\
\notag&-\sum_{r_1+r_2\leq r}\int_{\partial\tilde{\Sigma}_t}2h^{AB}\partial_A\partial^{r_2}_t\partial^{r_1}k_C{}^C\partial^{r_2}_t\partial^{r_1}k_{NB}\mathrm{vol}_{\partial\tilde{\Sigma}_t}
\end{align}
where in the last inequality, we applied Cauchy-Schwarz, trace inequality, \eqref{bootstrap} and the estimate \eqref{gijPhiest} in Lemma \ref{lem:gijPhiest}. 

{\it Note}: We need to be  careful of the boundary terms in the last line of \eqref{enineq2}, since they do not inherit any smallness. Instead, we perform a precise trace inequality and observe that the resulting constants are within a certain range that allows us to absorb the arising terms in the LHS. To begin with, it is necessary that we use the boundary condition \eqref{kNNbdcond}, in order to reduce the absolute value of the constants in each term generated by the procedure:
\begin{align}\label{bdtermsest4}
&-\int_{\partial\tilde{\Sigma}_t}2h^{AB}\partial_A\partial^{r_2}_t\partial^{r_1}k_C{}^C\partial^{r_2}_t\partial^{r_1}k_{NB}\mathrm{vol}_{\partial\tilde{\Sigma}_t}\\
\tag{by \eqref{kNNbdcond}}=&\int_{\partial\tilde{\Sigma}_t}h^{AB}\partial_A\partial^{r_2}_t\partial^{r_1}k_{NN}\partial^{r_2}_t\partial^{r_1}k_{NB}
-h^{AB}\partial_A\partial^{r_2}_t\partial^{r_1}k_C{}^C\partial^{r_2}_t\partial^{r_1}k_{NB}\mathrm{vol}_{\partial\tilde{\Sigma}_t}\\
\notag=&\int_{\tilde{\Sigma}_t}\bigg[h^{AB}\partial_AN\partial^{r_2}_t\partial^{r_1}k_{NN}\partial^{r_2}_t\partial^{r_1}k_{NB}+h^{AB}\partial_A\partial^{r_2}_t\partial^{r_1}k_{NN}N\partial^{r_2}_t\partial^{r_1}k_{NB}\\
\notag&-h^{AB}\partial_AN\partial^{r_2}_t\partial^{r_1}k_C{}^C\partial^{r_2}_t\partial^{r_1}k_{NB}+h^{AB}\partial_A\partial^{r_2}_t\partial^{r_1}k_C{}^CN\partial^{r_2}_t\partial^{r_1}k_{NB}\bigg]\mathrm{vol}_{\tilde{\Sigma}_t}\\
\notag&+\int_{\tilde{\Sigma}_t}\Gamma\star\partial\partial^{r_2}_t\partial^{r_1}k\star\partial^{r_2}_t\partial^{r_1}k\mathrm{vol}_{\tilde{\Sigma}_t}\\
\tag{IBP}\leq&\int_{\tilde{\Sigma}_t}\bigg[-h^{AB}N\partial^{r_2}_t\partial^{r_1}k_{NN}\partial_A\partial^{r_2}_t\partial^{r_1}k_{NB}+h^{AB}\partial_A\partial^{r_2}_t\partial^{r_1}k_{NN}N\partial^{r_2}_t\partial^{r_1}k_{NB}\\
&+h^{AB}N\partial^{r_2}_t\partial^{r_1}k_C{}^C\partial_A\partial^{r_2}_t\partial^{r_1}k_{NB}-h^{AB}\partial_A\partial^{r_2}_t\partial^{r_1}k_C{}^CN\partial^{r_2}_t\partial^{r_1}k_{NB}
\notag\bigg]\mathrm{vol}_{\tilde{\Sigma}_t}\\
\notag&+CC_0tE_k(t)+\varepsilon E_k(t)+\frac{C}{\varepsilon}E_k(0)\\
\leq&\int_{\tilde{\Sigma}_t}\frac{\eta}{2}|\nabla\partial^{r_2}_t\partial^{r_1}k_{NN}|_g^2+\frac{\eta}{2}|\nabla\partial^{r_2}_t\partial^{r_1}k_C{}^C|_g^2+\frac{1}{\eta}h^{AB}\partial^i \partial^{r_2}_t\partial^{r_1}k_{NA}\partial_i\partial^{r_2}_t\partial^{r_1}k_{NB}\mathrm{vol}_{\tilde{\Sigma}_t}\notag\\
\notag&+CC_0tE_k(t)+\varepsilon E_k(t)+\frac{C}{\varepsilon}E_k(0)
\end{align}
Incorporating \eqref{bdtermsest4} into \eqref{enineq2} and recalling the definition \eqref{Ek} of $E_k(t)$, we deduce the inequality:
\begin{align}\label{enineq3}
\frac{1}{2}E_k(t)\leq&\,\frac{C}{\varepsilon}E_{total}(0)+\frac{C}{\varepsilon}C_{b.d.}+CC_0\big[\frac{t}{\varepsilon}\sup_{\tau\in[0,t]}E_k(\tau) +tC_{b.d.}\big]+\varepsilon E_k(t)
+\max\{\frac{\eta}{2},\frac{1}{4\eta}\}E_k(t),
\end{align}
for all $t\in[0,T]$. Thus, setting $\eta=\frac{3}{4}$ and taking $\varepsilon,\frac{T}{\varepsilon}$ sufficiently small, we can absorb all $E_k$ terms in the LHS and deduce the estimate \eqref{Ekest}
\end{proof}
\begin{remark}\label{rem:trace}
We would like to emphasize the somehow surprising appearance of the boundary terms \eqref{bdtermsest4} in the final energy inequality and their seemingly delicate nature. If we had not splitted up the terms using \eqref{kNNbdcond}, then the coefficient of the last term in \eqref{enineq3} would be $\max\{\eta,\frac{1}{4\eta}\}\ge\frac{1}{2}$, which would render the corresponding term barely non-absorbable.
\end{remark}
\begin{proof}[Sketch of the Picard iteration scheme] One may construct a sequence of iterates $k^n,g^n,\Phi^n$, $n\in\mathbb{N}$, where $k^0,g^0,\Phi^0$ are set equal to their initial values everywhere, by considering the following {\it linear} system of equations:
\begin{align}
\notag\partial_tg_{ij}^{n+1}=&-2\Phi^nk_{ij}^n\\
\label{itsyst}[(\Phi^n)^{-1}\partial_t((\Phi^n)^{-1}\partial_t)-\Delta_{g^n}](k^{n+1})_i{}^j=&\,\mathcal{N}(\Phi_{1;0;2}^n,g^n_{0;0;2},k^n_{1;0;0},k^n_{0;0;1})_i{}^j\\
\notag\Delta_{g^n}\Phi^{n+1}=&\,|k^n|^2\Phi^{n+1}
\end{align}
where the RHS of the second equation corresponds to \eqref{boxkij2}.  Then, by assuming $g^n,k^n,\Phi^n$ satisfy the energy estimate \eqref{Etotalest}, imposing the boundary conditions\footnote{Here we write the boundary conditions for $k^{n+1}$ in covariant form, using the connection of $g^n$.}
\begin{align}\label{itbdcond}
\begin{split}
\Phi^{n+1}=1,\qquad(\tilde{k}^{n+1})_A{}^B=\text{tr}_{g^n}k^{n+1}=0,\qquad{}^{g^n}\nabla_Nk_{NA}^{n+1}=-{}^{g^n}\nabla_B(k^{n+1})_A{}^B,\qquad n\in\mathbb{N}\\
\frac{1}{2}[{}^{g^n}\nabla_N(k^{n+1})_{NN}-{}^{g^n}\nabla_N(k^{n+1})_A{}^A]=-{}^{g^n}\nabla^Ak_{NA}^{n+1},\qquad\text{on $\{\partial\tilde{\Sigma}_t\}_{t\in [0,T]}$,}
\end{split}
\end{align}
and by repeating the energy estimates just derived above,  we infer that $g^{n+1},k^{n+1},\Phi^{n+1}$ satisfy \eqref{Etotalest} as well. This proves by induction that the total energy of the iterates is uniformly bounded in $n\in\mathbb{N}$. As it is usual for quasilinear hyperbolic equations, see \cite[Appendix III.4.3.1]{ChoqBook}, in order to prove that the sequence is a contraction, we must close the corresponding estimates for the differences $g^{n+1}-g^n,k^{n+1}-k^n,\Phi^{n+1}-\Phi^n$ using the analogous energy containing one derivative less of the unknowns. This is due to the presence of terms of the form $(g^n-g^{n-1})\partial^2k^n$ in the resulting equations for the differences or terms of the form $(g^n-g^{n-1})\partial k^n$ in the boundary conditions satisfied by $k^{n+1}-k^n$. However, the energy argument is practically the same. As for the arising boundary integrals that correspond to the latter terms, they can be safely handled by trace inequality, as in \eqref{bdtermsest3}, provided we commute the equations with one coordinate derivative less than for the energy boundedness argument. Hence, we obtain local existence and uniqueness for the original reduced system \eqref{1stvar},\eqref{boxk},\eqref{Phieq}.

Finally, to be fully legitimate, we have to show that the above sequence of iterates is well-defined, that is, a solution to the linear system \eqref{itsyst}, subjected to the conditions \eqref{itbdcond}, actually exists. Given $\Phi^n,k^n$, the first equation is solved trivially for $g^{n+1}$ by integrating in $t$. For the third equation, we notice that the RHS has a favourable sign. Hence, existence for $\Phi^{n+1}-1$ with homogeneous Dirichlet boundary conditions can be shown via the standard Lax-Milgram argument. In order to obtain existence of a solution to the system of wave equations for the components of $k^{n+1}$, we may use a duality argument. This is mainly based on a priori estimates, which we derived in the previous subsection, and a study of the adjoint problem. Roughly speaking, if the adjoint system has a similar form, then the same energy estimates apply, yielding a weak solution by Riesz representation theorem, see \cite[\S5.2.2]{SarTig2}. One can then improve its regularity by using that of the initial/boundary data and the inhomogeneous terms (previous iterates).

Since the equations for $(\tilde{k}^{n+1})_A{}^B,\text{tr}_{g^n}k^{n+1}$, decouple from those for $(k^{n+1})_C{}^C-k_{NN}^{n+1},k_{NA}^{n+1}$, and they are subject to homogeneous Dirichlet boundary conditions, existence for the former is standard. The main ingredient for studying the adjoint system of $(k^{n+1})_C{}^C-k_{NN}^{n+1},k_{NA}^{n+1}$ is identifying the  dual boundary conditions. This is done by considering test functions $v_{CC},v_{NN},v_{NA}\in C^\infty_0(\{0<t<T\})$, multiplying the corresponding equations, integrating in $\{\tilde{\Sigma}_t\}_{t\in[0,T]}$, integrating by parts and setting the arising boundary integrals equal to zero. We may assume that $(k^{n+1})_C{}^C-k_{NN}^{n+1},k_{NA}^{n+1}$ have trivial data, by incorporating them in the RHS. Also, for convenience, we consider the covariant version of the equations, which take the form:
\begin{align}\label{covitboxk}
\begin{split}
(e_0^2-\Delta_{g^n})[(k^{n+1})_C{}^C-k_{NN}^{n+1}]=&\,F^1_{NN}\star \partial k^{n+1}+F^2_{NN}\star k^{n+1}+F^3_{NN}\\
(e_0^2-\Delta_{g^n})k_{NA}^{n+1}=&\,F^1_{NA}\star \partial k^{n+1}+F^2_{NA}\star k^{n+1}+F^3_{NA},
\end{split}
\end{align}
where $F^i_{Nj}$ are known, regular functions and $e_0:=\partial_s$. Then in order to make the boundary condition for $k^{n+1}_{NA}$, in \eqref{itbdcond}, purely homogeneous, we substract from $k^{n+1}_{NA}$ a function depending on the boundary data $(\hat{k}^{n+1})_A{}^B$, such that we have 
\begin{align}\label{itkNAbdcond}
{^{g^n}\nabla_Nk^{n+1}_{NA}}=-\frac{1}{2}\nabla_A(k^{n+1})_C{}^C
\end{align}
We consider now the weak formulation of the equations \eqref{covitboxk} that corresponds to a coercive energy [cf. \eqref{Ek}]:
\begin{align}\label{weakeq}
\notag&\int^T_0\int_{\tilde{\Sigma}_s}(v_C{}^C-v_{NN})(F^1_{NN}\star \partial k^{n+1}+F^2_{NN}\star k^{n+1}+F^3_{NN})\\
\notag&+8v_N{}^A(F^1_{NA}\star \partial k^{n+1}+F^2_{NA}\star k^{n+1}+F^3_{NA})
\mathrm{vol}_{\tilde{\Sigma}_s}ds\\
\notag=&\int^T_0\int_{\tilde{\Sigma}_s}(v_C{}^C-v_{NN})(e_0^2-\Delta_{g^n})[(k^{n+1})_C{}^C-k_{NN}^{n+1}]+8v_N{}^A(e_0^2-\Delta_{g^n})k_{NA}^{n+1}
\mathrm{vol}_{\tilde{\Sigma}_s}ds\\
=&\int^T_0\int_{\tilde{\Sigma}_s}[(k^{n+1})_C{}^C-k_{NN}^{n+1}](e_0^2-\Delta_{g^n})(v_C{}^C-v_{NN})+8k_{NA}^{n+1}(e_0^2-\Delta_{g^n})v_N{}^A
\mathrm{vol}_{\tilde{\Sigma}_s}ds\\
\notag&+\int^T_0\int_{\partial\tilde{\Sigma}_s}\bigg[[(k^{n+1})_C{}^C-k_{NN}^{n+1}]\nabla_N(v_C{}^C-v_{NN})-(v_C{}^C-v_{NN})\nabla_N[(k^{n+1})_C{}^C-k_{NN}^{n+1}]\\
\notag&+8k_{NA}^{n+1}\nabla_Nv_N{}^A
-8v_N{}^A\nabla_Nk_{NA}^{n+1}\bigg]
\mathrm{vol}_{\partial\tilde{\Sigma}_s}ds
\end{align}
The adjoint system of equations for $v_C{}^C-v_{NN},v_{NA}$ can be read from \eqref{weakeq}, by setting the boundary terms in the last two lines equal to zero. Evidently, it is of the form \eqref{covitboxk} and the dual boundary conditions are derived from
\begin{align}\label{dualbdcond}
\notag0=&\int^T_0\int_{\partial\tilde{\Sigma}_s}\bigg[[(k^{n+1})_C{}^C-k_{NN}^{n+1}]\nabla_N(v_C{}^C-v_{NN})-(v_C{}^C-v_{NN})\nabla_N[(k^{n+1})_C{}^C-k_{NN}^{n+1}]\\
\notag&+8k_{NA}^{n+1}\nabla_Nv_N{}^A
-8v_N{}^A\nabla_Nk_{NA}^{n+1}\bigg]
\mathrm{vol}_{\partial\tilde{\Sigma}_s}ds\\
\notag=&\int^T_0\int_{\partial\tilde{\Sigma}_s}\bigg[2(k^{n+1})_C{}^C\nabla_N(v_C{}^C-v_{NN})-2v_C{}^C2\nabla^Ak_{NA}^{n+1}
+8k_{NA}^{n+1}\nabla_Nv_N{}^A\\
&\notag+8v_N{}^A\frac{1}{2}\nabla_A(k^{n+1})_C{}^C\bigg]
\mathrm{vol}_{\partial\tilde{\Sigma}_s}ds\\
=&\int^T_0\int_{\partial\tilde{\Sigma}_s}\bigg[2(k^{n+1})_C{}^C(\nabla_Nv_C{}^C-\nabla_Nv_{NN}-2\nabla_Av_N{}^A)\\
\notag&+4k_{NA}^{n+1}(2\nabla_Nv_N{}^A+\nabla^Av_C{}^C)\bigg]
\mathrm{vol}_{\partial\tilde{\Sigma}_s}ds
\end{align}
Thus, we conclude that the boundary conditions for $v_C{}^C-v_{NN},v_{NA}$ are the same as the ones in \eqref{itbdcond},\eqref{itkNAbdcond} for $(k^{n+1})_C{}^C-k^{n+1}_{NN},k^{n+1}_{NA}$:
\begin{align}\label{dualbdcond2}
\nabla_Nv_C{}^C-\nabla_Nv_{NN}=2\nabla^Av_{NA},\qquad \nabla_Nv_{NA}=-\frac{1}{2}\nabla_Av_C{}^C.
\end{align}
A posteriori this can be justified from the fact that the boundary conditions for $k^{n+1}$ were chosen such that after various integrations by parts in the energy estimates, the quadratic terms in first derivatives of $k^{n+1}$ in the boundary terms cancel out, leaving terms that can be handled by trace inequality.
The energy estimates for the adjoint problem are therefore the same. This proves existence for $(k^{n+1})_C{}^C-k^{n+1}_{NN},k^{n+1}_{NA}$ and completes our sketch of the Picard iteration scheme.
\end{proof}

\section{Vanishing of the Ricci tensor: Solution to the EVE}\label{sec:verEVE}

Let ${\bf g}$ be the metric of the form \eqref{metric}, where $g$ satisfies \eqref{1stvar} and  $k,\Phi$ solve \eqref{boxk},\eqref{Phieq}. 
Then, according to Proposition \ref{prop:equiveq} , the spacetime Ricci tensor satisfies the propagation equation \eqref{e0Rij4}. Considering the Einstein tensor $G_{ij}={\bf R}_{ij}-\frac{1}{2}g_{ij}{\bf R}$, \eqref{e0Rij4} implies the equation:
\begin{align}\label{e0Gij}
e_0G_{ij}=\nabla_i\mathcal{G}_j+\nabla_j\mathcal{G}_i-g_{ij}\nabla^a\mathcal{G}_a-\nabla_i\nabla_j\mathrm{tr}k+\frac{1}{2}g_{ij}\Delta_g\mathrm{tr}k+k_{ij}{\bf R}+\frac{1}{2}g_{ij}e_0{\bf R}_{00}-g_{ij}k^{ab}{\bf R}_{ab}
\end{align}
The above equation is coupled to the wave equation \eqref{boxtrk} for $\mathrm{tr}k$:
\begin{align}\label{boxtrk2}
e_0^2\mathrm{tr}k-\Delta_g \mathrm{tr}k
=e_0[(\mathrm{tr}k)^2]+4\Phi^{-1}(\nabla^a\Phi)\mathcal{G}_a
\end{align}
On the other hand, from \eqref{d/dttrk} and \eqref{Phieq} we also have
\begin{align}\label{R00}
{\bf R}_{00}=e_0\mathrm{tr}k.
\end{align}
Combining \eqref{boxtrk2}-\eqref{R00}, the equation \eqref{e0Gij} becomes
\begin{align}\label{e0Gij2}
\begin{split}
e_0G_{ij}=&\,\nabla_i\mathcal{G}_j+\nabla_j\mathcal{G}_i-g_{ij}\nabla^a\mathcal{G}_a-\nabla_i\nabla_j\mathrm{tr}k+g_{ij}\Delta_g\mathrm{tr}k+k_{ij}{\bf R}-g_{ij}k^{ab}{\bf R}_{ab}\\
&+\frac{1}{2}g_{ij}e_0[(\mathrm{tr}k)^2]+2g_{ij}\Phi^{-1}(\nabla^a\Phi)\mathcal{G}_a,
\end{split}
\end{align}
where we can also write
\begin{align}\label{Ric=G}
{\bf R}=-G,\qquad{\bf R}_{ab}=G_{ij}-\frac{1}{2}g_{ij}G,\qquad G:={\bf g}^{ab}G_{ab}.
\end{align}
Next, we utilise the contracted second Bianchi identity to derive a propagation equation for $\mathcal{G}_i$:
\begin{align}\label{BianchiR0i}
\notag e_0\mathcal{G}_i=e_0{\bf R}_{0i}=&\,D_0{\bf R}_{0i}+\Phi^{-1}\nabla^j\Phi{\bf R}_{ji}-k_i{}^j{\bf R}_{0j}+\Phi^{-1}\nabla_i\Phi{\bf R}_{00}\\
\notag=&\,D_j{\bf R}_i{}^j-\frac{1}{2}\partial_i{\bf R}+\Phi^{-1}\nabla^j\Phi{\bf R}_{ji}-k_i{}^j{\bf R}_{0j}+\Phi^{-1}\nabla_i\Phi{\bf R}_{00}\\
=&\,\nabla_j{\bf R}_i{}^j-\frac{1}{2}\partial_i{\bf R}+\Phi^{-1}\nabla^j\Phi{\bf R}_{ji}+k_i{}^j{\bf R}_{0j}-\Phi^{-1}\nabla_i\Phi{\bf R}_{00}\\
\notag&+\mathrm{tr}k{\bf R}_{0i}+k_i{}^j{\bf R}_{0j}\\
\notag=&\,\nabla_jG_i{}^j+\Phi^{-1}\nabla^j\Phi{\bf R}_{ji}+2k_i{}^j{\bf R}_{0j}+\mathrm{tr}k{\bf R}_{0i}-\Phi^{-1}\nabla_i\Phi{\bf R}_{00}
\end{align}
Taking the divergence of \eqref{e0Gij} and utilising \eqref{BianchiR0i}, we derive the following wave equation for $\mathcal{G}_i$:
\begin{align}\label{boxR0i}
 e_0^2\mathcal{G}_i
-\Delta_g\mathcal{G}_i=
L(\nabla G,G,\nabla\mathcal{G},\mathcal{G},\nabla\nabla\mathrm{tr}k,e_0\nabla\mathrm{tr}k,\nabla\mathrm{tr}k,e_0\mathrm{tr}k)
\end{align}
where $L$ is a linear operator in the corresponding variables.\footnote{The coefficients can be expressed in terms of the reduced solution and are therefore  in $C^1$.}

\subsection{Boundary conditions}

The conditions \eqref{kNNbdcond}-\eqref{kCCbdcond}, combined with the Coddazi identity \eqref{Codazzi}, imply the following boundary conditions on $\mathcal{T}$: 
\begin{align}\label{Ricbdcond}
\begin{split}
\mathrm{tr}k=0,\qquad \mathcal{G}_A={\bf R}_{0A}=\partial_A\mathrm{tr}k-\nabla^ik_{Ai}=\partial_A\mathrm{tr}k=0,\\
 \mathcal{G}_N={\bf R}_{0N}=\nabla_N\mathrm{tr}k-\nabla^ik_{Ni}=\nabla_Nk_A{}^A-\nabla^Ak_{NA}=\frac{1}{2}\nabla_N\mathrm{tr}k
\end{split} 
\end{align}
These conditions, together with \eqref{e0Gij},\eqref{boxtrk2},\eqref{boxR0i}, define a well-determined boundary value problem for $\mathrm{tr}k,{\bf G}_{ij},\mathcal{G}_i$.

\subsection{Modified equations}

In order to avoid losing derivatives, we need to modify the $\mathcal{G}_i$'s, such that they all have homogeneous Dirichlet data on the boundary. For this purpose we set:
\begin{align}\label{Gitilde}
\tilde{\mathcal{G}}_i:=\mathcal{G}_i-\frac{1}{2}\nabla_i\mathrm{tr}k
\end{align}
Then \eqref{e0Gij} becomes
\begin{align}\label{e0Gij2}
\begin{split}
e_0G_{ij}=&\,\nabla_i\tilde{\mathcal{G}}_j+\nabla_j\tilde{\mathcal{G}}_i-g_{ij}\nabla^a\tilde{\mathcal{G}}_a+\frac{1}{2}g_{ij}\Delta_g\mathrm{tr}k+k_{ij}{\bf R}-g_{ij}k^{ab}{\bf R}_{ab}\\
&+\frac{1}{2}g_{ij}e_0[(\mathrm{tr}k)^2]+2g_{ij}\Phi^{-1}(\nabla^a\Phi)\tilde{\mathcal{G}}_a+g_{ij}\Phi^{-1}(\nabla^a\Phi)\nabla_a\mathrm{tr}k
\end{split}
\end{align}
On the other hand, $\tilde{\mathcal{G}}_i$ satisfies a wave equation of the same form as \eqref{boxR0i}, since the additional terms in the LHS, after plugging in \eqref{Gitilde}, equal
\begin{align}\label{addtermGitilde}
-\frac{1}{2}e_0^2\nabla_i\mathrm{tr}k+\frac{1}{2}\Delta_g\nabla_i\mathrm{tr}k,
\end{align}
which after the appropriate commutations and the use of \eqref{boxtrk2} give rise to terms which are already included in $\tilde{L}$ below:
\begin{align}\label{boxGitilde}
e_0^2\tilde{\mathcal{G}}_i
-\Delta_g\tilde{\mathcal{G}}_i
=\tilde{L}(\nabla G,G,\nabla\tilde{\mathcal{G}},\tilde{\mathcal{G}},\nabla\nabla\mathrm{tr}k,e_0\nabla\mathrm{tr}k,\nabla\mathrm{tr}k,e_0\mathrm{tr}k)
\end{align}
The advantage of working with $\tilde{\mathcal{G}}_i$ is that they all satisfy homogeneous Dirichlet boundary conditions:
\begin{align}\label{tildeGibdcond}
\tilde{\mathcal{G}}_i:=\mathcal{G}_i-\frac{1}{2}\nabla_i\mathrm{tr}k={\bf R}_{0i}-\frac{1}{2}\nabla_i\mathrm{tr}k\overset{\eqref{Ricbdcond}}{=}0,&&\text{on $\mathcal{T}$}.
\end{align}
\subsection{Energy estimates for the Einstein tensor}

Recall that the Einstein tensor and $\mathrm{tr}k,\partial_t\text{tr}k$ vanish initially on $\tilde{\Sigma}_0$ by assumption, see \eqref{dtkinit}-\eqref{dttrkinit}. In order to prove the vanishing of the Einstein tensor, together with the validity of our gauge condition everywhere, $\mathrm{tr}k=0$, we need to establish an energy estimate for the system \eqref{e0Gij2},\eqref{boxtrk2},\eqref{boxGitilde}, subject to the boundary conditions \eqref{tildeGibdcond}, $\mathrm{tr}k=0$.

The energy for $\tilde{\mathcal{G}}_i$ reads
\begin{align}\label{Gien}
\int_{\tilde{\Sigma}_t}g^{ij}e_0\tilde{\mathcal{G}}_ie_0\tilde{\mathcal{G}}_j+\nabla^j\tilde{\mathcal{G}}^i\nabla_j\tilde{\mathcal{G}}_i+\tilde{\mathcal{G}}^i\mathcal{G}_i
\end{align}
Going back to \eqref{boxtrk2}, we notice that there is room to commute once with any tangential derivative $\partial=\partial_t,\partial_1,\partial_2$. Thus, by a standard Gronwall type of argument, we can control the following energy of $\mathrm{tr}k$:
\begin{align}\label{Etrk}
E[\text{tr}k]:=\int_{\tilde{\Sigma}_t}(e_0^2\mathrm{tr}k)^2+(\nabla\partial\mathrm{tr}k)^2+(e_0\partial\mathrm{tr}k)^2+(e_0\mathrm{tr}k)^2+(\nabla\mathrm{tr}k)^2+(\mathrm{tr}k)^2
\end{align}
by
\begin{align}\label{trkenest}
E[\text{tr}k]
\lesssim T\sum_i\sup_{t\in[0,T]}(\|e_0\tilde{\mathcal{G}}_i\|_{L^2}^2+\|\nabla\tilde{\mathcal{G}}_i\|_{L^2}^2+\|\tilde{\mathcal{G}}_i\|_{L^2}^2
\end{align}
Then, using the wave equation \eqref{boxtrk2}, we can also control $\nabla_N\nabla_N\text{tr}k$ in $L^2$, hence, controlling all second derivatives of $\text{tr}k$:
\begin{align}\label{trkenest2}
\|\nabla_N\nabla_N\text{tr}k\|_{L^2}^2
\lesssim T\sum_i\sup_{t\in[0,T]}(\|e_0\tilde{\mathcal{G}}_i\|_{L^2}^2+\|\nabla\tilde{\mathcal{G}}_i\|_{L^2}^2)+\sum_i\sup_{t\in[0,T]}\|\tilde{\mathcal{G}}_i\|_{L^2}^2
\end{align}
Note that although the last term in the preceding inequality, seemingly, has no smallness in $T$, we can directly bound it from its time derivative
\begin{align}\label{GiL2est}
e_0\|\tilde{\mathcal{G}}_i\|_{L^2}^2\lesssim \|\tilde{\mathcal{G}}_i\|^2_{L^2}+\|e_0\tilde{\mathcal{G}}_i\|^2_{L^2}\quad\Rightarrow\quad \|\tilde{\mathcal{G}}_i\|_{L^2}^2\lesssim T\|e_0\tilde{\mathcal{G}}_i\|^2_{L^2}
\end{align}

On the other hand, we are forced to estimate $G_{ij}$ only in $L^2$, due to the first order terms $\nabla_j\tilde{\mathcal{G}}_i$ in the RHS of \eqref{e0Gij2}. A direct time differentiation of its $L^2$ norm, followed by a use of Gronwall's inequality yields
\begin{align}\label{Gijenest}
\sum_{i,j=1}^3\int_{\tilde{\Sigma}_t}(G_{ij})^2
\lesssim T\sum_i\sup_{t\in[0,T]}(\|e_0\tilde{\mathcal{G}}_i\|_{L^2}^2+\|\nabla\tilde{\mathcal{G}}_i\|_{L^2}^2+\|\tilde{\mathcal{G}}_i\|_{L^2}^2)
\end{align}

For the energy estimates for $\tilde{\mathcal{G}}_i$, we first note that both $\tilde{\mathcal{G}}_i,e_0\tilde{\mathcal{G}}_i$ vanish initially at $t=0$ by \eqref{R0iinit} and \eqref{dtkinit},\eqref{BianchiR0i}. Also, according to \eqref{tildeGibdcond}, $\tilde{\mathcal{G}}_i$ has homogeneous Dirichlet boundary data, for all $i$. Thus, if it was not for the terms $\nabla G$ in the RHS of \eqref{boxGitilde}, the energy estimates for $\mathcal{G}_a$ would be standard. However, we can handle these terms by integration by parts. More precisely, the standard energy estimate for \eqref{boxR0i}, combined with \eqref{trkenest}-\eqref{Gijenest}, gives an inequality of the form:
\begin{align}\label{R0iest}
\begin{split}
\int_{\tilde{\Sigma}_t}(e_0\tilde{\mathcal{G}})^2+(\nabla\tilde{\mathcal{G}})^2+(\tilde{\mathcal{G}})^2
\leq &\,CT\sum_a\sup_{t\in[0,T]}(\|e_0\tilde{\mathcal{G}}_a\|_{L^2}^2+\|\nabla\tilde{\mathcal{G}}_a\|_{L^2}^2+\|\tilde{\mathcal{G}}_a\|_{L^2}^2)\\
&+\int^t_0\int_{\tilde{\Sigma}_\tau}
L(\nabla G,\partial_\tau\tilde{\mathcal{G}})
\end{split}
\end{align}
Moreover, integrating by parts we have:
\begin{align}\label{R0iest2}
\begin{split}
\int^t_0\int_{\tilde{\Sigma}_\tau}f\nabla G\partial_\tau\tilde{\mathcal{G}}_b
\overset{\eqref{tildeGibdcond}}{=}&-\int^t_0\int_{\tilde{\Sigma}_\tau}( fG\partial_\tau\nabla\tilde{\mathcal{G}}_b+f\partial_\tau\Gamma G\tilde{\mathcal{G}}_b+\nabla f G\partial_\tau\tilde{\mathcal{G}}_b)\\
=&\int^t_0\int_{\tilde{\Sigma}_\tau}[\partial_\tau (fG)\nabla\tilde{\mathcal{G}}_b-\Gamma G\tilde{\mathcal{G}}_b-\nabla f G\partial_\tau\tilde{\mathcal{G}}_b]-\int_{\tilde{\Sigma}_t}f G\nabla\tilde{\mathcal{G}}_b
\end{split}
\end{align}
Notice that all the terms in the preceding RHS can be absorbed in the LHS of \eqref{R0iest}, after plugging in \eqref{e0Gij2}, by using Cauchy-Schwarz and the smallness in $T$ in the above estimates.

Thus, all variables $\mathrm{tr}k,G_{ij},\tilde{\mathcal{G}}_i$ must vanish everywhere, which shows that the solution of the reduced equations \eqref{1stvar},\eqref{boxk},\eqref{Phieq} is indeed a solution of the EVE, satisfying the maximal gauge $\text{tr}k=0$.



\end{document}